\documentclass[letterpaper,11pt,oneside,reqno]{amsart}

\usepackage{amsmath,amssymb,amsthm,amsfonts}
\usepackage{hyperref}
\usepackage{graphicx,color,colortbl}
\usepackage{upgreek}
\usepackage[mathscr]{euscript}

\allowdisplaybreaks
\numberwithin{equation}{section}

\usepackage{tikz}
\usetikzlibrary{
	shapes,
	arrows,
	positioning,
	decorations.markings,
	decorations.pathmorphing
}

\usepackage{array}
\usepackage{adjustbox}
\usepackage{cleveref}
\usepackage{enumerate}

\usepackage[DIV=13]{typearea}

\usepackage[shadow]{todonotes}
\renewcommand{\Re}{\mathop{\mathrm{Re}}}
\renewcommand{\Im}{\mathop{\mathrm{Im}}}

\newtheorem{proposition}{Proposition}[section]
\newtheorem{lemma}[proposition]{Lemma}
\newtheorem{corollary}[proposition]{Corollary}
\newtheorem{theorem}[proposition]{Theorem}
\theoremstyle{definition}
\newtheorem{definition}[proposition]{Definition}
\newtheorem{remark}[proposition]{Remark}

\begin{document}
\title{PushTASEP in inhomogeneous space}
\author{Leonid Petrov}

\date{}

\maketitle

\begin{abstract}
We consider the PushTASEP (pushing totally asymmetric simple exclusion process, also sometimes called long-range TASEP) with the step initial configuration evolving in an inhomogeneous space. That is, the rate of each particle's jump depends on the location of this particle. We match the distribution of the height function of this PushTASEP with Schur processes. Using this matching and determinantal structure of Schur processes, we obtain limit shape and fluctuation results which are typical for stochastic particle systems in the Kardar-Parisi-Zhang universality class. PushTASEP is a close relative of the usual TASEP. In inhomogeneous space the former is integrable, while the integrability of the latter is not known.
\end{abstract}

\section{Introduction and main results}
\label{sec:intro}

\subsection{Overview}

The totally asymmetric simple exclusion process
(TASEP) was introduced 
about 50 years ago, independently in 
biology \cite{macdonald1968bioASEP},
\cite{MacdonaldGibbsASEP1969}
and probability theory
\cite{Spitzer1970}.
The latter paper also introduced zero-range processes
and long-range TASEP.
The long-range TASEP
(which we call \emph{PushTASEP} following more recent works)
is the focus of the present paper.

Since early works, TASEP and PushTASEP 
were often studied in parallel.
Once a result (such as description of hydrodynamics and local equilibria
\cite{Liggett1985}, limiting density \cite{Rost1981}, 
or asymptotic fluctuations \cite{johansson2000shape}) for TASEP is established, 
it can often be generalized to PushTASEP using similar tools. 
See \cite{derrida1991dynamics} 
for hydrodynamics and related results for the PushTASEP
(viewed as a special case of the Toom’s model),
and, e.g.,
\cite{diekerWarren2008determinantal}, 
\cite{BorFerr08push}
for fluctuation results.
Borodin and Ferrari
\cite{BorFerr2008DF}
introduced a two-dimensional 
stochastic
particle system
whose 
two different one-dimensional (marginally Markovian)
projections are TASEP and PushTASEP.
This coupling 
works best for special examples of initial data, most notably, for 
step initial configurations.
It is worth pointing out that most known asymptotic fluctuation
results for TASEP, PushTASEP, and related systems (in the Kardar-Parisi-Zhang universality
class, cf. \cite{CorwinKPZ}, \cite{QuastelSpohnKPZ2015})
require
\emph{integrability}, that is, the presence of some
exact formulas in the pre-limit system. 

Running either TASEP or PushTASEP in inhomogeneous space
(such that the particles' jump rates depend on their locations)
is a natural generalization. Hydrodynamic approach 
works well in macroscopically inhomogeneous systems, and
allows to write down PDEs for limiting densities
\cite{Landim1996hydrodynamics}, 
\cite{seppalainen1999existence},
\cite{rolla2008last},
\cite{Seppalainen_Discont_TASEP_2010},
\cite{calder2015directed}. 
This leads to law of large numbers type results for 
the height function (in particular, of the inhomogeneous TASEP).
However, when the disorder is microscopic (such as just one slow bond),
this affects the local equilibria, and makes the analysis of both
limit shape and asymptotic fluctuations of TASEP
much harder
\cite{Basuetal2014_slowbond},
\cite{basu2017invariant}.
Overall, putting TASEP on inhomogeneous space 
breaks its integrability.

On the other hand, considering particle-dependent inhomogeneity
(when the jump rate depends on the particle's number, but not its location)
in TASEP preserves its integrability, and allows to extract the corresponding fluctuation
results, cf. \cite{BaikBBPTASEP},
\cite{BorodinEtAl2009TwoSpeed},
\cite{Duits2011GFF}.

\medskip

The main goal of this paper is to show that, in contrast with TASEP,
the PushTASEP in inhomogeneous space 
started from the step initial configuration
retains the integrability
for arbitrary inhomogeneities. 
Namely, we obtain a matching of the PushTASEP to a certain Schur process,
which follows by taking a \emph{third} marginally Markovian projection of the
two-dimensional dynamics of Borodin--Ferrari
\cite{BorFerr2008DF}, which was not observed previously.
This coupling is also present in the 
Robinson-Schensted-Knuth insertion --- a mechanism originally employed to 
obtain TASEP fluctuations in \cite{johansson2000shape}.
The coupling of inhomogeneous PushTASEP to Schur processes, together with their determinantal structure
\cite{okounkov2001infinite}, \cite{okounkov2003correlation}
leads to exact formulas for the PushTASEP.
We illustrate the integrability 
by obtaining limit shape and fluctuation results for PushTASEP with arbitrary
macroscopic inhomogeneity.

\begin{remark}
	Based 
	on the tools employed in the present work, 
	one can even say that
	our results could have been observed already in the mid-2000s.
	However, it is the much more recent development of stochastic vertex models,
	especially their couplings in \cite{BorodinBufetovWheeler2016},
	\cite{BufetovMatveev2017},
	\cite{BufetovPetrovYB2017} to Hall-Littlewood processes, that prompted the present work
	(as a $t=0$ degeneration of the Hall-Littlewood situation).
	Asymptotic behavior of the Hall-Littlewood deformation of the PushTASEP
	(in a homogeneous case) was studied in \cite{Ghosal2017KPZ}.
\end{remark}

Other examples of 
integrable stochastic particle systems in one-dimensional inhomogeneous
space have been recently studied in 
\cite{BorodinPetrov2016Exp},
\cite{SaenzKnizelPetrov2018}.
These systems may be viewed as analogues of 
$q$-TASEP or TASEP, respectively in continuous space. 
(The $q$-TASEP is a certain integrable $q$-deformation of TASEP
\cite{BorodinCorwin2011Macdonald}.) In those inhomogeneous systems, a certain 
choice of inhomogeneity leads to interesting phase transitions
corresponding to formation of ``traffic jams'', when the density goes to infinity.
In PushTASEP the density is bounded by one, and so we do not expect this type of
phase transitions to appear.
A two-dimensional stochastic particle system 
in inhomogeneous space 
unifying both the inhomogeneous PushTASEP considered in the present paper,
and a TASEP-like process similar to the one in
\cite{SaenzKnizelPetrov2018},
is studied in 
\cite{theodoros2019_determ} (the latter 
was completed simultaneously with the present paper and independently of it).

\medskip

In the rest of the introduction we give the main definitions and formulate the results.

\subsection{PushTASEP in inhomogeneous space}
\label{sub:PushTASEP_intro}

Fix a positive \emph{speed function} $\xi_\bullet=\{\xi_x\}_{x\in \mathbb{Z}_{\ge1}}$, uniformly
bounded away from zero and infinity. 
By definition, the PushTASEP is a continuous time Markov
process on particle configurations 
\begin{equation}\label{eq:PushTASEP_configurations}
	\mathsf{x}_1(t)<\mathsf{x}_2(t)<\mathsf{x}_3(t)<\ldots
\end{equation}
on $\mathbb{Z}_{\ge1}$ (at most one particle per site is allowed).
We consider only the \emph{step initial configuration}
$\mathsf{x}_i(0)=i$ for all $i\ge1$,
so at all times the particle configuration has a leftmost particle. 

The system evolves as follows. At each
site $x\in \mathbb{Z}_{\ge1}$ there is an independent exponential clock with
rate $\xi_x$ (i.e., the mean waiting time till the clock rings is
$1/\xi_x$). When the clock at site $x\in\mathbb{Z}_{\ge1}$ rings and there is
no particle at $x$, nothing happens. Otherwise, let some particle
$\mathsf{x}_i$ be at $y$. When the clock rings, $\mathsf{x}_i$ jumps to the
right by one. If the destination $x+1$ is occupied by $\mathsf{x}_{i+1}$, then
$\mathsf{x}_{i+1}$ is pushed to the right by one, which may trigger subsequent
instantaneous pushes. That is, if there is a packed cluster of particles to
the right of $\mathsf{x}_i$, i.e.,
$\mathsf{x}_{i+m}-m=\ldots=\mathsf{x}_{i+1}-1=\mathsf{x}_i$ and
$\mathsf{x}_{i+m+1}-1>\mathsf{x}_{i+m}$ for some
$m\in \left\{
	1,2,\ldots
\right\}
\cup
\left\{
	+\infty
\right\}$,
then each of the
particles $\mathsf{x}_{i+\ell}$, $\ell=1,\ldots, m$, is instantaneously pushed
to the right by one. The case $m=+\infty$ corresponds to pushing 
the whole right-infinite densely packed cluster of particles to the right by one. Clearly, 
the evolution preserves the order \eqref{eq:PushTASEP_configurations} of
particles. See \Cref{fig:inhom_Push} for an illustration.

\begin{figure}[htpb]
	\centering
	\includegraphics[width=.7\textwidth]{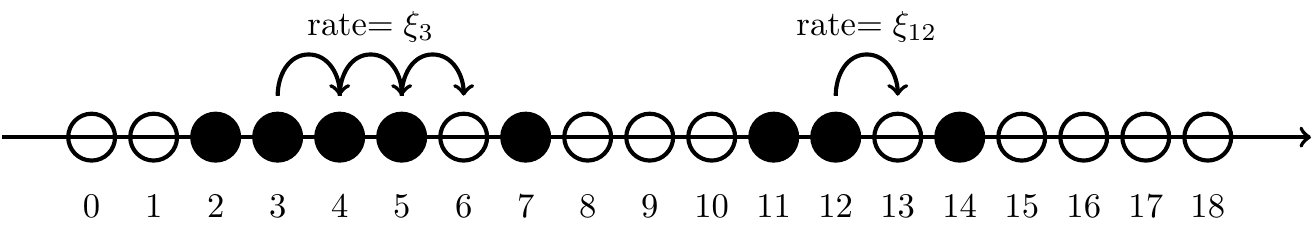}
	\caption{%
		Two possible transitions 
		in the inhomogeneous PushTASEP. 
		The first one corresponds to activating the particle at $3$
		which happens at rate $\xi_3$. 
		This particle then pushes two other particles (located at $4$ and $5$) 
		to the right by one.
		The second transition corresponds to activating the particle at $12$
		at rate $\xi_{12}$.%
	}
	\label{fig:inhom_Push}
\end{figure}

Thus described Markov process on particle configurations in
$\mathbb{Z}_{\ge1}$ is well-defined. Indeed, consider its restriction to
$\left\{ 1,\ldots, N \right\}\subset \mathbb{Z}_{\ge1}$ for any $N$. This is a
continuous time Markov process on a finite space in which at most one
exponential clock can ring at a given time moment. For different $N$, such
restrictions are compatible, and thus the process on configurations in
$\mathbb{Z}_{\ge1}$ exists by the Kolmogorov extension theorem.
Note however that the number of jumps in the process on 
$\mathbb{Z}_{\ge1}$ during each initial time interval
$[0,\varepsilon)$, $\varepsilon>0$, 
is infinite.

\begin{remark}
	The PushTASEP on the whole space $\mathbb{Z}_{\ge1}$
	might be alternatively described as follows.
	When a particle jumps, it goes to the closest empty site that is to
	its right. If there are no empty sites to the right, the particle disappears.
\end{remark}

\subsection{Determinantal structure}

The \emph{height function} of the PushTASEP is
defined as
\begin{equation}
	\label{eq:height_function_def}
	h(t,N):=\#\{\textnormal{particles
		in $\left\{ 1,\ldots,N\right\}$ at time $t$}\},
	\qquad N\in \mathbb{Z}_{\ge1}, \quad t\in \mathbb{R}_{\ge0}.
\end{equation}
The step initial condition corresponds to $h(0,N)=N$ for all $N\ge0$.

\begin{definition}
	\label{def:down_right_path}
	A \emph{down-right path} 
	in the $(t,N)$ plane is a collection
	$\mathfrak{p}=\{(t_i,N_i)\}_{i=1}^{r}$, where $r\ge1$, 
	\begin{equation}
		\label{eq:space-like-path-def}
		N_1\ge N_2\ge \ldots \ge N_r\ge0,\qquad 0\le t_1\le t_2\le
		\ldots \le t_r,
	\end{equation}
	and the points $(t_i,N_i)$ are pairwise distinct.
\end{definition}

\begin{remark}
	Down-right paths are also called \emph{space-like} (as opposed to
	\emph{time-like}, when both $t_i$ and $N_i$ increase). These names come from a
	growth model reformulation, cf. \cite{derrida1991dynamics}, \cite{Ferrari_Airy_Survey}.
	\label{rmk:space-like}
\end{remark}

Define a kernel depending on the 
speed function $\xi_{\bullet}$ by
\begin{equation}
	\label{eq:K_intro}
	K(t,N,x;t',N',x'):=
	\mathbf{1}_{t=t'}\mathbf{1}_{N=N'}\mathbf{1}_{x=x'}
	-
	\frac{1}{(2\pi \mathbf{i})^2}
	\oint\oint \frac{dz\,dw}{z-w}\frac{w^{x'+N'}}{z^{x+N+1}}
	\,e^{t z-t'w}
	\frac{\prod_{a=1}^{N}\bigl(z-\xi_a\bigr)}
	{\prod_{b=1}^{N'}\bigl(w-\xi_b\bigr)}.
\end{equation}
The integration contours are positively oriented simple closed curves around $0$, the
$w$ contour additionally encircles all $\{\xi_x\}_{x\in\mathbb{Z}_{\ge1}}$, and
the contours satisfy $|z|>|w|$ for $t\le t'$ and $|z|<|w|$ for $t>t'$.
(Throughout the text $\mathbf{1}_{A}$ stands for the
indicator of $A$, which is $1$ if condition $A$ is true, 
and is $0$ otherwise. By $\mathbf{1}$ without subscripts we will also
mean the identity operator.)

Fix a down-right path $\mathfrak{p}=\{(t_i,N_i)\}_{i=1}^{r}$,
and define the space
\begin{equation*}
	\mathcal{X}:=
	\mathcal{X}_1\sqcup \ldots \sqcup\mathcal{X}_r,
	\qquad \mathcal{X}_i =\mathbb{Z}.
\end{equation*}
For $y\in \mathcal{X}_i$ set $t(y)=t_i$, $N(y)=N_i$.

\begin{definition}\label{def:Lp_point_process}
	We define a determinantal random point process\footnote{For general
	definitions and properties of determinantal processes see, e.g.,
	\cite{Soshnikov2000}, \cite{peres2006determinantal}, or \cite{Borodin2009}.} 
	$\mathfrak{L}_{\mathfrak{p}}$
	on $\mathcal{X}$
	with the correlation kernel expressed through $K$ of \eqref{eq:K_intro}.
	Namely, for any $m\ge1$ and any pairwise distinct
	$y^{1},\ldots y^{m}\in \mathcal{X}$, let the corresponding correlation
	function of $\mathfrak{L}_{\mathfrak{p}}$ be given by 
	\begin{equation}\label{eq:K_correlation_functions_on_copies_of_Z}
		\mathbb{P}
		\left( \textnormal{$\mathfrak{L}_{\mathfrak{p}}$ contains all of $y^1,\ldots,y^m $} \right)
		=
		\mathop{\mathrm{det}}_{i,j=1}^{m}
		\Bigl[ 
			K\bigl(t(y^i),N(y^i),y^i; t(y^j),N(y^j),y^j\bigr)
		\Bigr].
	\end{equation}
	The process $\mathfrak{L}_{\mathfrak{p}}$ exists
	because it corresponds to 
	column lengths in a certain specific
	Schur process, see \Cref{sec:detSchurStructure}
	for details. The Schur process interpretation
	also implies that
	on each $\mathcal{X}_i=\mathbb{Z}$
	the random point configuration
	$\mathfrak{L}_{\mathfrak{p}}$
	almost surely has a
	leftmost point.
	Denote it by $\widehat{\ell}(t_i,N_i)$.
\end{definition}
The joint distribution of the leftmost points
$\{\widehat{\ell}(t_i,N_i)\}$
is identified with the
inhomogeneous PushTASEP. The following theorem 
is the main structural result of the present paper.
\begin{theorem}
	\label{thm:push_TASEP_multipoint_correlations_intro}
	Fix an arbitrary down-right path $\mathfrak{p}=\{(t_i,N_i)\}_{i=1}^{r}$.
	The joint distribution of the 
	PushTASEP height function along this down-right path 
	is related to the determinantal process
	$\mathfrak{L}_{\mathfrak{p}}$ defined above as
	\begin{equation*}
		\bigl\{
			h(t_i,N_i)
		\bigr\}_{i=1}^{r}
		\stackrel{d}{=}
		\bigl\{
			N_i+\widehat{\ell}(t_i,N_i)
		\bigr\}_{i=1}^{r},
	\end{equation*}
	where 
	``\,$\stackrel{d}{=}$''
	means equality in distribution.
\end{theorem}
\begin{corollary}
	\label{cor:pushTASEP_Fredholm_single_point_formula_intro}
	For any $t\ge0$, $N\ge1$, and $y\ge0$ we have
	\begin{equation}
	\begin{split}
		\mathbb{P}\bigl( h(t,N)>y \bigr)
		&=
		\det
		\bigl(
			\mathbf{1}-K(t,N,\cdot;t,N,\cdot)
		\bigr)
		_{ \left\{ \ldots,y-N-2,y-N-1,y-N \right\}}
		\\&=
		1+\sum_{n=1}^{\infty}\frac{(-1)^{n}}{n!}
		\sum_{x_1=-\infty}^{y-N} \ldots \sum_{x_n=-\infty}^{y-N} 
		\;
		\mathop{\mathrm{det}}_{i,j=1}^{n}
		\left[
			K(t,N,x_i;t,N,x_j) 
		\right],
	\end{split}
	\label{eq:Fredhom_theorem_intro}
	\end{equation}
	where the second equality is the series expansion of 
	the Fredholm determinant given in the first equality. (See \Cref{appsub:Fredholm_dets}
	below for more details on Fredholm determinants.) 
	Similar Fredholm determinantal formulas are available for 
	joint distributions of the PushTASEP height function
	along down-right paths.
\end{corollary}

\Cref{thm:push_TASEP_multipoint_correlations_intro}
is a restatement of a known result on how Schur processes appear in stochastic
interacting particle systems in $(2+1)$ dimensions. Via
Robinson-Schensted-Knuth (RSK) correspondences, such connections can be
traced to \cite{Vershik1986}, and they were heavily 
utilized in probabilistic context
starting from \cite{baik1999distribution}, \cite{johansson2000shape},
\cite{PhahoferSpohn2002}. 
Markov dynamics on particle configurations 
coming from RSK
correspondences were studied in \cite{Baryshnikov_GUE2001},
\cite{OConnell2003}, \cite{OConnell2003Trans}. 
Another type of Markov dynamics
whose fixed-time distributions are given by Schur processes was
introduced in \cite{BorFerr2008DF}, and it, too, can be utilized to obtain
\Cref{thm:push_TASEP_multipoint_correlations_intro}. 
A self-contained exposition of the proof
of this theorem following the latter approach is presented in
\Cref{sec:detSchurStructure}.

\begin{remark}[Connection to vertex models]
	Yet another alternative way of getting 
	\Cref{thm:push_TASEP_multipoint_correlations_intro}
	is to view the PushTASEP 
	as a degeneration of the stochastic six vertex
	model \cite{GwaSpohn1992}, \cite{BCG6V}.
	The latter was recently connected to 
	Hall-Littlewood processes
	\cite{borodin2016stochastic_MM}, 
	\cite{BorodinBufetovWheeler2016},
	\cite{BufetovMatveev2017}.
	Setting the Hall-Littlewood parameter 
	$t$ to zero leads to a distributional mapping between 
	our inhomogeneous PushTASEP and Schur processes.
\end{remark}

\subsection{Hydrodynamics}
\label{sub:hydrodynamics}

Let the space and time in the PushTASEP, as well as the speed function
scale as follows:
\begin{equation}
	\label{eq:space-time-speed-scaling}
	t=\tau L , 
	\qquad 
	N=\lfloor \eta L \rfloor ,
	\qquad 
	\xi_x= \boldsymbol\xi(x/L), \quad x\in \mathbb{Z}_{\ge1},
\end{equation}
where $L$ is the large parameter going to infinity, and
$\boldsymbol\xi(\cdot)$ is a fixed positive
\emph{limiting speed function}
bounded away from zero and infinity.
Under \eqref{eq:space-time-speed-scaling}, one expects that the height function 
$h(t,N)$ admits a limit shape (i.e., law of large numbers type) behavior
of the form
\begin{equation}
	\label{eq:expected_limit_shape_type_behavior}
	\frac{h(\tau L,\lfloor \eta L \rfloor )}{L}\to \mathfrak{h}(\tau,\eta),
	\quad 
	\text{in probability as $L\to+\infty$}.
\end{equation}
Let us first write down a
partial differential equation for the limiting 
density 
\begin{equation*}
	\uprho(\tau,\eta):=\frac{\partial}{\partial\eta}\mathfrak{h}(\tau,\eta)
\end{equation*}
using hydrodynamic arguments as in
\cite{Andjel1984}, \cite{Rezakhanlou1991hydrodynamics}, 
\cite{Landim1996hydrodynamics},
\cite{Seppalainen_Discont_TASEP_2010}.
We do not rigorously justify this equation, 
but rather check that the density coming from 
fluctuation analysis satisfies the hydrodynamic equation
(this check is performed in \Cref{sec:app_check_hydrodynamics}).

Because of our scaling \eqref{eq:space-time-speed-scaling},
locally around every scaled location $\lfloor \eta L \rfloor$
the behavior of the PushTASEP (when we zoom at the lattice level)
is homogeneous
with constant speed $\xi=\boldsymbol\xi(\eta)$. 
Thus, locally around $\lfloor \eta L \rfloor$
the PushTASEP configuration
should have\footnote{We do not rigorously justify this claim here.} 
a particle distribution on $\mathbb{Z}$
which is invariant under shifts of $\mathbb{Z}$
and is stationary under the speed $\xi$ homogeneous PushTASEP dynamics
on the whole line.

A classification of translation invariant stationary distributions 
for the homogeneous PushTASEP on $\mathbb{Z}$
is available \cite{guiol1997resultat},
\cite{andjel2005long}.
Namely, ergodic (=~extreme)
such measures are precisely the 
Bernoulli product measures.
For the Bernoulli product measure of density $\rho\in[0,1]$, the 
flux (=~current) of particles in the PushTASEP (i.e., the 
expected number of particles crossing a given bond
in unit time interval) is readily seen to be 
$j(\rho)=\dfrac{\xi\rho}{1-\rho}$.
Therefore, the partial differential equation 
for the limiting density should have the form:
\begin{equation}
	\label{eq:hydrodynamic_PDE}
	\frac{\partial}{\partial\tau}\uprho(\tau,\eta)
	+
	\frac{\partial}{\partial\eta}
	\left( 
		\frac{\boldsymbol\xi(\eta)\uprho(\tau,\eta)}
		{1- \uprho(\tau,\eta)} 
	\right)=0,
	\qquad 
	\uprho(0,\eta)=\mathbf{1}_{\eta\ge0}.
\end{equation}
The singularity at $\tau=0$ coming from the initial data corresponds to the
fact that the PushTASEP makes an infinite number of jumps during every time
interval $[0,t]$, $t>0$. That is, from $t=0$ to $t=\tau L$ (for every $\tau>0$
in the regime $L\to+\infty$), the density of the particles drops below $1$
everywhere.

\begin{remark}
	\label{rmk:homogeneous_case_formulas}
	One sees from, e.g., \cite[Claim 3.1 and Proposition 3.2]{BorFerr2008DF}
	that 
	in the homogeneous case $\boldsymbol\xi(\eta)\equiv 1$, 
	a solution to \eqref{eq:hydrodynamic_PDE}
	has the form
	\begin{equation}
		\label{eq:homogeneous_limit_shape_solution}
		\uprho(\tau,\eta)=
		\begin{cases}
			1-\sqrt{{\tau}/{\eta}},&\eta\ge\tau;\\
			0,&0\le\eta<\tau.
		\end{cases}
	\end{equation}
	The condition $\eta\ge\tau$ for nonzero density 
	comes from the behavior of the leftmost particle 
	in the PushTASEP which performs a simple random walk.
	Integrating this density in $\eta$ gives the limiting height function
	$\mathfrak{h}(\tau,\eta)=\left( \sqrt\eta-\sqrt\tau \right)^2$, $\eta\ge\tau$.
\end{remark}
Next, we present a solution to \eqref{eq:hydrodynamic_PDE} for general
$\boldsymbol\xi(\cdot)$.

\subsection{Limit shape}
\label{sub:limit_shape}

For any $\eta>0$, set 
\begin{equation}
	\label{eq:tau_e_of_eta_defn}
	\tau_e=\tau_e(\eta):=
	\int_0^\eta \frac{dy}{\boldsymbol\xi(y)},
\end{equation}
this is the rescaled time when the leftmost particle in the 
PushTASEP reaches $\lfloor \eta L \rfloor $.
Consider the following equation in $z$:
\begin{equation}
	\label{eq:critical_solution_t_equation}
	\tau=\int_0^\eta \frac{\boldsymbol\xi(y)}{\left( {z}- \boldsymbol\xi(y) \right)^2}\,dy.
\end{equation}

\begin{lemma}
	\label{lemma:z_critical_solution_using_t_equation}
	For any $\eta>0$ and $\tau\in(0,\tau_e(\eta))$ equation \eqref{eq:critical_solution_t_equation}
	has a unique root $z$
	on the negative real line.
\end{lemma}
We denote this solution by
$\mathfrak{z}=\mathfrak{z}(\tau,\eta)$.
\begin{proof}[Proof of \Cref{lemma:z_critical_solution_using_t_equation}]
	This is evident due to the strict increasing of the right-hand side of 
	\eqref{eq:critical_solution_t_equation} in $z\in(-\infty,0)$,
	and the fact that at $z=0$ this right-hand side is equal to $\tau_e(\eta)$ given by \eqref{eq:tau_e_of_eta_defn}.
\end{proof}

\begin{definition}
	\label{def:limit_shape_h}
	Define $\mathfrak{h}=\mathfrak{h}(\tau,\eta)$ by
	\begin{equation}
		\label{eq:h_limit_shape_definition}
		\mathfrak{h}(\tau,\eta):=
		\begin{cases}
			\displaystyle\int_0^\eta\frac{\mathfrak{z}^2(\tau,\eta)}
			{\left( \mathfrak{z}(\tau,\eta)- \boldsymbol\xi(y) \right)^2}
			\,dy,& 0\le\tau\le \tau_e(\eta);\\
			0,&\tau\ge\tau_e(\eta),
		\end{cases}
	\end{equation}
	where $\mathfrak{z}(\tau,\eta)$ comes from \Cref{lemma:z_critical_solution_using_t_equation}.
	We call $\mathfrak{h}$ the \emph{limiting height function}.
\end{definition}
\begin{remark}
	\label{rmk:continuity_in_eta}
	\textbf{1.}
	Since the right-hand side of \eqref{eq:critical_solution_t_equation}
	depends on $z$ and $\eta$ in a continuous way
	when $z\le0$, the function $\eta\mapsto\mathfrak{z}(\tau,\eta)$ 
	is continuous for each fixed $\tau$.
	Thus, the height function \eqref{eq:h_limit_shape_definition}
	is also continuous in $\eta$.
	(This continuity extends to 
	the unique 
	$\eta_e$ such that $\tau_e(\eta_e)=\tau$
	because both cases in 
	\eqref{eq:h_limit_shape_definition} give zero.)

	\medskip
	\noindent
	\textbf{2.}
	Equivalently, the function $\tau\mapsto
	\mathfrak{h}(\tau,\eta)$ 
	is the Legendre dual
	of the function
	$\displaystyle
	z\mapsto F(z,\eta):=\int_0^\eta\frac{z}{\boldsymbol\xi(y)-z}\,dy$,
	where $z<0$.
	That is, $\displaystyle\mathfrak{h}(\tau,\eta)=\max_{z<0}(\tau z-F(z,\eta))$.
\end{remark}

One can check
that the limiting density 
corresponding to $\mathfrak{h}(\tau,\eta)$ 
(defined as 
$\uprho(\tau,\eta)=\frac{\partial}{\partial\eta}\mathfrak{h}(\tau,\eta)$
when this derivative exists)
is 
expressed through $\mathfrak{z}$ as
\begin{equation}
	\label{eq:rho_via_z_critical_point}
	\uprho(\tau,\eta)=\frac{\mathfrak{z}(\tau,\eta)}{\mathfrak{z}(\tau,\eta)- \boldsymbol\xi(\eta)},
	\qquad 0\le\tau\le \tau_e(\eta).
\end{equation}
One can also verify that
$\uprho(\tau,\eta)$ formally satisfies the hydrodynamic equation
\eqref{eq:hydrodynamic_PDE}.
This is done in \Cref{sec:app_check_hydrodynamics}.
See \Cref{fig:limsh,fig:limsh2} for illustrations of limit shapes of the height function
and the density.

\begin{figure}[htpb]
	\centering
	\begin{tabular}{c|c|c|c}
		&
		$\boldsymbol\xi(x)\equiv 1$, $\tau=1$
		&
		$\boldsymbol\xi(x)=\mathbf{1}_{\eta<3}+5\cdot \mathbf{1}_{\eta\ge3}$, $\tau=1$
		&
		$\boldsymbol\xi(x)=\mathbf{1}_{\eta<3}+\frac12\cdot \mathbf{1}_{\eta\ge3}$, $\tau=1$
		\\\hline
		\raisebox{40pt}{$\uprho$}&\includegraphics[width=.25\textwidth]{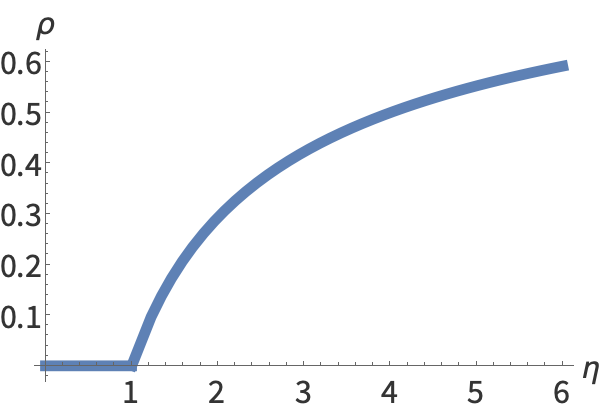}
		\rule{0pt}{83pt}
		&
		\includegraphics[width=.25\textwidth]{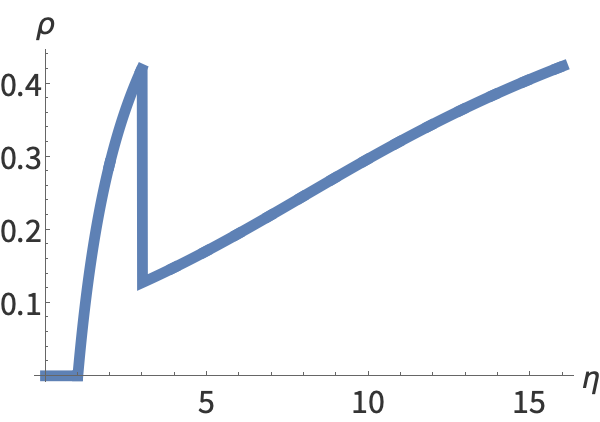}
		&
		\includegraphics[width=.25\textwidth]{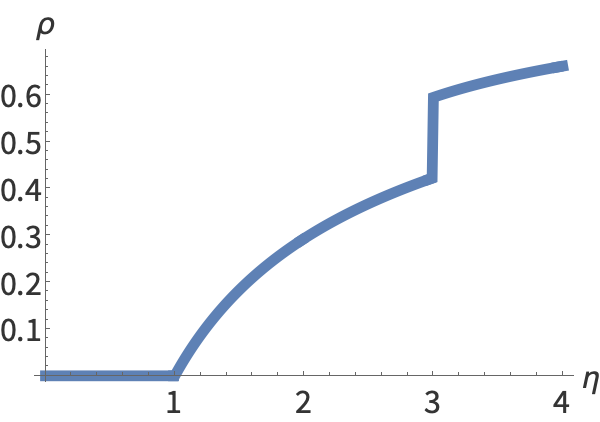}
		\\\hline
		\raisebox{40pt}{$\mathfrak{h}$}&
		\includegraphics[width=.25\textwidth]{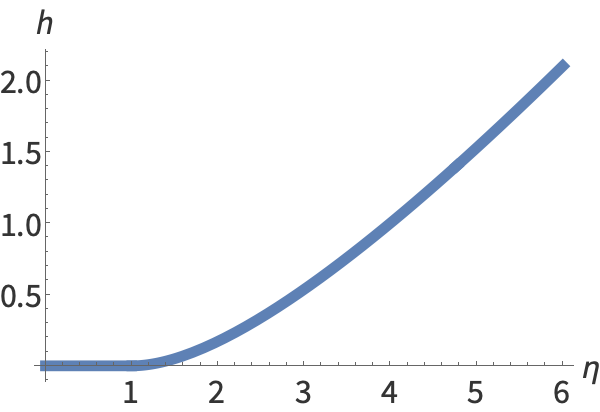}
		\rule{0pt}{83pt}
		&
		\includegraphics[width=.25\textwidth]{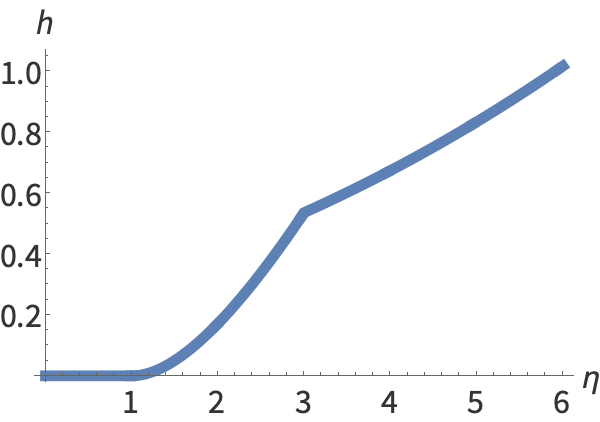}
		&
		\includegraphics[width=.25\textwidth]{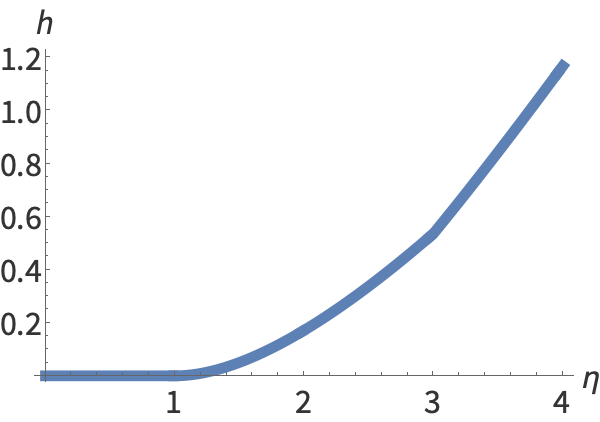}
	\end{tabular}
	\caption{Limiting density and height function for three cases of piecewise
	linear $\boldsymbol\xi(\cdot)$.}
	\label{fig:limsh}
\end{figure}

\begin{figure}[htpb]
	\centering
	\includegraphics[width=.4\textwidth]{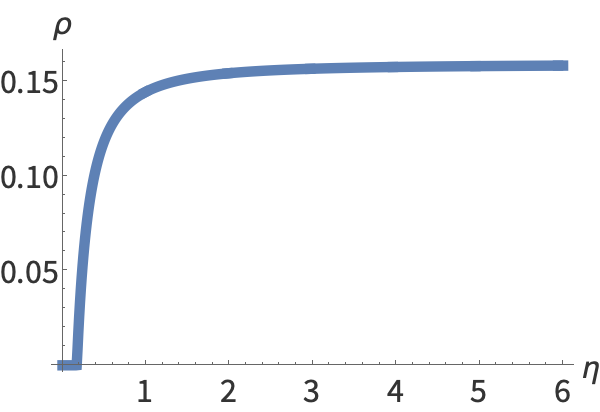}
	\qquad 
	\includegraphics[width=.4\textwidth]{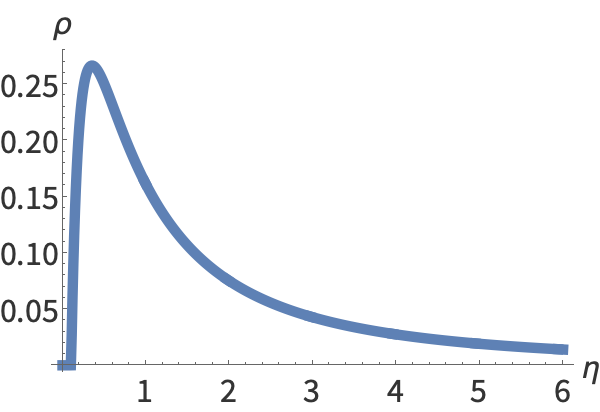}
	\caption{Two more examples of limiting density, for $\boldsymbol\xi(x)=x+\frac1{10}$ (left)
	and $\boldsymbol\xi(x)=x^2+\frac1{10}$ (right) on the interval $[0,6]$.
	These numerical computations suggest that many particles
	run off to infinity --- a positive proportion (left) or all of them (right).
	Note that in both cases the speed function $\boldsymbol\xi(\cdot)$ is unbounded.}
	\label{fig:limsh2}
\end{figure}

\begin{theorem}[Limit shape]
	\label{thm:limit_shape_result}
	Fix arbitrary $\tau,\eta>0$.
	If
	the limiting speed function $\boldsymbol\xi(\cdot)$ 
	is piecewise continuously differentiable
	on $[0,\eta]$,
	then in the regime \eqref{eq:space-time-speed-scaling}
	we have the convergence
	$L^{-1}h(\tau L,\lfloor \eta L \rfloor )\to \mathfrak{h}(\tau,\eta)$
	in probability as $L\to+\infty$. Here $h$ is the random
	height function of our PushTASEP, 
	and $\mathfrak{h}$ is defined by \eqref{eq:h_limit_shape_definition}.
\end{theorem}
\Cref{thm:limit_shape_result}
follows from 
the fluctuation result 
(\Cref{thm:single_point_fluctuations}) which 
is formulated next.

\subsection{Asymptotic fluctuations}
\label{sub:limit_distributions}

Using the notation of \Cref{sub:limit_shape},
define for $0<\tau<\tau_e(\eta)$:
\begin{equation}
	\label{eq:d_variance_in_TW}
	\mathfrak{d}=\mathfrak{d}(\tau,\eta):=
	\left( 
		\int_0^\eta 
		\frac{\mathfrak{z}^2(\tau,\eta)\boldsymbol\xi(y)}
		{\left( \boldsymbol\xi(y)-\mathfrak{z}(\tau,\eta) \right)^3}\,dy
	\right)^{\frac13}>0.
\end{equation}
Note that this quantity is also continuous in $\eta$ 
similarly to \Cref{rmk:continuity_in_eta}.1.
The following is the main asymptotic fluctuation result of the 
present paper:
\begin{theorem}
	\label{thm:single_point_fluctuations}
	Fix arbitrary $\eta>0$
	and 
	$\tau\in(0,\tau_e(\eta))$.
	If
	the limiting speed function $\boldsymbol\xi(\cdot)$ 
	is piecewise continuously differentiable on $[0,\eta]$,
	then in the regime \eqref{eq:space-time-speed-scaling}
	we have the convergence
	\begin{equation}
		\label{eq:result_one-point-TW-convergence_intro}
		\lim_{L\to+\infty}
		\mathbb{P}
		\left( \frac{h(\tau L,\lfloor \eta L \rfloor )-L \mathfrak{h}(\tau,\eta)}
		{
			\mathfrak{d}(\tau,\eta)
			L^{1/3}
		}
			>-r \right)
		=
		F_{GUE}(r),\qquad 
		r\in \mathbb{R},
	\end{equation}
	where $F_{GUE}$ is the GUE Tracy--Widom distribution 
	\cite{tracy_widom1994level_airy}.
\end{theorem}

This result implies the law of large numbers 
(\Cref{thm:limit_shape_result}).

Using the determinantal structure of \Cref{thm:push_TASEP_multipoint_correlations_intro},
it is possible to also obtain 
(under slightly more restrictive smoothness conditions on $\boldsymbol\xi(\cdot)$)
multipoint asymptotic fluctuation results along space-like paths.
These fluctuations 
are governed by the top line of the Airy$_2$ line ensemble \cite{PhahoferSpohn2002}. 
We will not focus on this result as it is a standard (by now)
extension of the single-point fluctuations of \Cref{thm:single_point_fluctuations}.
The extension readily follows from the determinantal structure together with a
double contour integral kernel.
For example, see 
\cite{ferrari2003step} or
\cite{BorFerr08push},
\cite{Borodin2008TASEPII}
for such computations for random tilings and TASEP-like particle systems, respectively.

\begin{remark}
	\Cref{thm:single_point_fluctuations}
	means that the insertion of inhomogeneity does not
	affect the fluctuation behavior of the PushTASEP
	compared to the
	homogeneous case.
	On the other hand, the inhomogeneity we consider in this paper \eqref{eq:space-time-speed-scaling}
	is relatively ``mild'' --- it varies only macroscopically but not
	microscopically, and also is bounded (so that
	behavior like
	\Cref{fig:limsh2} is out of the present scope).
	It would be interesting to see if less regular inhomogeneity
	might lead to different fluctuation behavior, whether in the regime 
	of \Cref{thm:single_point_fluctuations}, or around the 
	``edge'' $(\tau_e(\eta)L,\lfloor  \eta L \rfloor )$.
	At this edge in the homogeneous case one sees 
	fluctuations on the scale
	$L^{1/2}$
	described by the largest
	eigenvalues of GUE random matrices,
	and this behavior should persist in the presence of ``mild'' inhomogeneity~\eqref{eq:space-time-speed-scaling}.
\end{remark}

\subsection{Outline}

In \Cref{sec:detSchurStructure} we describe the connection between the inhomogeneous
PushTASEP and Schur processes, and establish
\Cref{thm:push_TASEP_multipoint_correlations_intro}. 
In 
\Cref{sec:limit_shape_single_fluctuations}
we perform asymptotic analysis
and 
establish \Cref{thm:limit_shape_result,thm:single_point_fluctuations}
on the limit shape and asymptotic fluctuations.
In
\Cref{sec:app_check_hydrodynamics} we check that the limiting density $\uprho$
defined in \Cref{sub:limit_shape} formally satisfies the hydrodynamic equation
coming from the inhomogeneous PushTASEP.

\subsection{Acknowledgments}

I am grateful to Konstantin Matveev for an insightful remark on connections
with RSK column insertion, and to Alexei Borodin, Alexey Bufetov,
Patrik Ferrari, Alisa Knizel, and Axel Saenz for helpful discussions. 
I am also grateful to anonymous referees for valuable suggestions.
The work was partially supported by
the NSF grant DMS-1664617.

\section{Schur processes and inhomogeneous PushTASEP}
\label{sec:detSchurStructure}

Here we present a self-contained proof of
\Cref{thm:push_TASEP_multipoint_correlations_intro} which follows from
results on Schur processes
\cite{okounkov2003correlation} and the two-dimensional stochastic particle dynamics
introduced in \cite{BorFerr2008DF}.

\subsection{Young diagrams}
\label{appsub:Young_diagrams}

A partition is a nonincreasing integer sequence of the form 
$\lambda=(\lambda_1\ge\ldots\ge\lambda_{\ell(\lambda)}>0)$. The number of nonzero
parts $\ell(\lambda)$ (which must be finite) is called the length of a
partition. Partitions are represented by \emph{Young diagrams}, such that
$\lambda_1,\lambda_2,\ldots $ are lengths of the successive rows. The column
lengths of a Young diagram are denoted by $\lambda_1'\ge\lambda_2'\ge\ldots$. 
They form the \emph{transposed} Young diagram $\lambda'$. See
\Cref{fig:YD_example} for an illustration.

\begin{figure}[htpb]
	\centering
	\includegraphics[width=.2\textwidth]{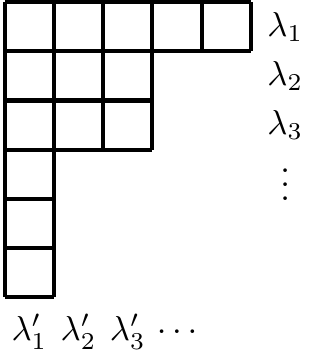}
	\caption{A Young diagram $\lambda=(5,3,3,1,1,1)$ for which the
		transposed diagram is $\lambda'=(6,3,3,1,1)$.}
	\label{fig:YD_example}
\end{figure}

\subsection{Schur functions}
\label{appsub:Schur_functions}

For each Young diagram $\lambda$, let $s_\lambda$ be the corresponding Schur
symmetric function \cite[Ch. I.3]{Macdonald1995}. Evaluated at $N$
variables $u_1,\ldots,u_N$ (where $N\ge \ell(\lambda)$ is arbitrary),
$s_\lambda$ becomes the symmetric polynomial
\begin{equation}
	\label{eq:schur_def}
	s_\lambda(u_1,\ldots,u_N) =
	\frac{\det[u_i^{\lambda_j+N-j}]_{i,j=1}^N}{\det[u_i^{N-j}]_{i,j=1}^N}.
\end{equation}
If $N<\ell(\lambda)$, then $s_\lambda(u_1,\ldots,u_N )=0$ by definition. When
all $u_i\ge0$, the value $s_\lambda(u_1,\ldots,u_N )$ is also nonnegative.

Along with evaluating Schur functions at finitely many variables, we also need
their \emph{Plancherel specializations} defined as
\begin{equation*}
	s_\lambda(\mathrm{Pl}_t):=
	\lim_{K\to+\infty}s_\lambda\Big( 
		\underbrace{
			\frac{t}{K},\ldots,\frac{t}{K} 
		}_{\textnormal{$K$ times}} 
	\Big),
	\qquad 
	t\in \mathbb{R}_{\ge0}.
\end{equation*}
This limit exists for every $\lambda$ (for example, see 
\cite[Section 2.1.4]{okounkov2001infinite}). 
It can be expressed through the number of standard Young
tableaux of shape~$\lambda$, which is the same as the dimension of the
corresponding irreducible representation of the symmetric group of order 
$\lambda_1+\lambda_2+\ldots$. The values
$s_\lambda(\mathrm{Pl}_t)$ are nonnegative for all $t\ge0$. When $t=0$, we have
$s_\lambda(\mathrm{Pl}_0)=\mathbf{1}_{\lambda=\varnothing}$.

The Schur functions satisfy Cauchy summation identities. We will need 
the following version:
\begin{equation*}
	\sum_{\lambda}s_\lambda(u_1,\ldots,u_N )s_\lambda(\mathrm{Pl}_t)=
	e^{t(u_1+\ldots+u_N )} ,
\end{equation*}
where the sum runs over all Young diagrams. However, summands
corresponding to $\ell(\lambda)>N$ vanish.
There are also skew Schur symmetric functions $s_{\lambda/\mu}$ which may be
defined 
as expansion coefficients as follows
(since Schur polynomials form a linear basis in the space of symmetric
polynomials in the corresponding variables, we expand $s_\lambda$ as a 
symmetric polynomial in $u_{N+1},\ldots,u_{N+M} $):
\begin{equation*}
	s_{\lambda}(u_1,\ldots,u_{N+M} )=
	\sum_{\mu}s_{\lambda/\mu}(u_1,\ldots,u_N )\,s_{\mu}(u_{N+1},\ldots,u_{N+M} ).
\end{equation*}
The function $s_{\lambda/\mu}$ vanishes unless the Young diagram $\lambda$ contains
$\mu$ (notation: $\lambda\supset\mu$). Skew Schur functions satisfy skew modifications of the Cauchy summation
identity. They also admit Plancherel specializations, and, moreover,
$s_{\lambda/\mu}(\mathrm{Pl}_t)$ is expressed through the number of 
standard tableaux of the skew shape $\lambda/\mu$. We refer to, e.g.,
\cite[Ch I.5]{Macdonald1995} for details.

\subsection{Schur processes}
\label{appsub:Schur_processes}

Here we recall the definition (at appropriate level of generality) of Schur
processes introduced in \cite{okounkov2003correlation}. Let $\xi_\bullet$ be a speed
function as in \Cref{sub:PushTASEP_intro}, and take a down-right path $\{(t_i,N_i)\}_{i=1}^{r}$ 
(\Cref{def:down_right_path}). A \emph{Schur process} associated with these data
is a probability distribution on sequences $(\lambda;\mu)$ of Young diagrams
(see \Cref{fig:Schur_process} for an illustration)
\begin{equation}
	\label{eq:Schur_process_diagrams}
	\varnothing=\mu^{(1)}\subset \lambda^{(1)}\supset \mu^{(2)}\subset
	\lambda^{(2)}\supset\mu^{(3)}\subset \ldots \subset
	\lambda^{(r-1)}\supset\mu^{(r)}=\varnothing
\end{equation}
with probability weights
\begin{equation}
	\label{eq:Schur_process_weights}
	\begin{split}
		\mathsf{SP}(\lambda;\mu)&=
		\frac{1}{Z_{\mathsf{SP}}}
		\prod_{i=1}^{r-1}s_{\lambda^{(i)}/\mu^{(i)}}(\mathrm{Pl}_{t_{i+1}-t_i})
		\prod_{j=1}^{r-1}s_{\lambda^{(j)}/\mu^{(j+1)}}(\xi_{(N_{j+1},N_{j}]})
		\\
		&=
		\frac{1}{Z_{\mathsf{SP}}}\,
		s_{\lambda^{(1)}/\mu^{(1)}}(\mathrm{Pl}_{t_2-t_1})
		s_{\lambda^{(1)}/\mu^{(2)}}(\xi_{(N_2,N_1]})
		s_{\lambda^{(2)}/\mu^{(2)}}(\mathrm{Pl}_{t_3-t_2}) \ldots \\&\hspace{150pt}\times
		s_{\lambda^{(r-1)}/\mu^{(r-1)}}(\mathrm{Pl}_{t_{r}-t_{r-1}})
		s_{\lambda^{(r-1)}/\mu^{(r)}}(\xi_{(N_{r},N_{r-1}]}).
	\end{split}
\end{equation}
Here $\xi_{(a,b]}$ for $a\le b$ means the string $(\xi_{a+1},\ldots ,\xi_b)$.
Note that some of the specializations above can be empty.
The normalizing
constant in \eqref{eq:Schur_process_weights} is
\begin{equation*}
	Z_{\mathsf{SP}}=\exp\biggl\{ \sum_{i=2}^{r}t_i \bigl(
	\xi_{N_i+1}+\ldots+\xi_{N_{i-1}}\bigr) \biggr\} ,
\end{equation*}
which is computed using the skew Cauchy identity.

\begin{figure}[htpb]
	\centering
	\includegraphics[width=.5\textwidth]{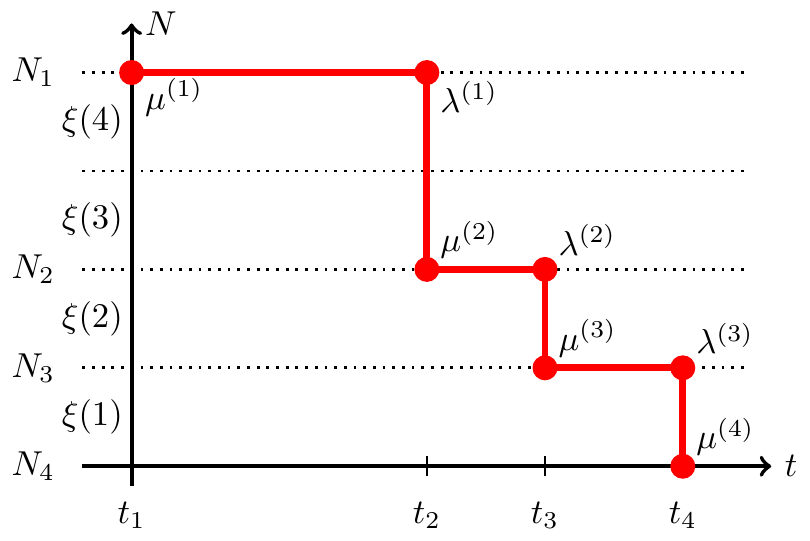}
	\caption{An illustration of the Schur process \eqref{eq:Schur_process_weights} 
	corresponding to a down-right path with $r=4$.
	For convenience we take $t_1=N_r=0$ so that the 
	corresponding Young diagrams are almost sure empty under
	the Schur process.}
	\label{fig:Schur_process}
\end{figure}

The marginal distribution of any $\lambda^{(i)}$ under the Schur
process \eqref{eq:Schur_process_weights} is a \emph{Schur measure}
\cite{okounkov2001infinite} whose probability weights are
\begin{equation}\label{eq:Schur_measure_xis}
	\mathsf{SM}(\lambda^{(i)})=
	e^{-t_{i+1}(\xi_1+\ldots +\xi_{N_i})}
	s_{\lambda^{(i)}}\bigl(\xi_1,\xi_2,\ldots,\xi_{N_{i}}\bigr)
	s_{\lambda^{(i)}}(\mathrm{Pl}_{t_{i+1}}).
\end{equation}

\subsection{Correlation kernel}
\label{appsub:Schur_correlations}

As shown in \cite[Theorem 1]{okounkov2003correlation}, the Schur process such as
\eqref{eq:Schur_process_weights} can be interpreted as a determinantal random
point process, and its correlation kernel is expressed as a double contour
integral. To recall this result, consider the particle configuration
\begin{equation}
	\label{eq:particle_configuration_for_Schur_process}
	\Bigl\{
	\lambda^{(i)}_j-j \colon
	i=1,\ldots,r-1,\ j=1,2,\ldots \Bigr\} \subset
	\underbrace{\mathbb{Z}\times\ldots\times\mathbb{Z}} _{\textnormal{$r-1$
			times}}
\end{equation}
corresponding to a sequence \eqref{eq:Schur_process_diagrams}
(where we sum over all the $\mu^{(j)}$'s).
The configurations $\lambda^{(i)}_j-j$, $j\ge1$, are infinite
and are densely packed at $-\infty$ (i.e., we append each 
$\lambda^{(i)}$ by infinitely many zeroes).
Then for any $m$ and any pairwise distinct locations $(l_i,x_i)$,
$i=1,\ldots,m$, where $1\le l_i\le r-1$ and $x_i\in \mathbb{Z}$, we have
\begin{equation*}
	\begin{split}
		&\mathbb{P}\Bigl( \textnormal{there are points of the configuration
		\eqref{eq:particle_configuration_for_Schur_process} at each of the locations
		$(l_i,x_i)$} \Bigr)
		\\&\hspace{280pt}
		=\det \left[ K_{\mathsf{SP}}(l_i,x_i;l_j,x_j) \right]_{i,j=1}^{m}.
	\end{split}
\end{equation*}
The kernel $K_{\mathsf{SP}}$ has the form
\begin{equation}
	\label{eq:kernel_Schur_general}
	K_{\mathsf{SP}}(l,x;l',y)= \frac{1}{(2\pi \mathbf{i})^2}\oint\oint
	\frac{dz\,dw}{z-w}\frac{w^{y}}{z^{x+1}}\frac{\Phi(l,z)}{\Phi(l',w)},
\end{equation}
where
\begin{equation*}
	\Phi(l,z)=e^{z t_{l+1}}\prod_{i=1}^{N_{l}}(1-z^{-1}\xi_i).
\end{equation*}
The integration contours in 
\eqref{eq:kernel_Schur_general}
are positively oriented simple closed curves around $0$,
the contour $w$ in addition encircles $\{\xi_x\}_{x\in\mathbb{Z}_{\ge1}}$,
and on these contours $|z| > |w|$ for $l\le l'$ and $|z| < |w|$ for $l > l'$.

\subsection{Coupling PushTASEP and Schur processes}
\label{appsub:Schur_dynamics}

Fix a speed function $\xi_\bullet$ as above. We will consider (half continuous
Schur) \emph{random fields of Young diagrams}
\begin{equation*}
	\{\lambda^{(t,N)}\colon t\in \mathbb{R}_{\ge0}, N\in
	\mathbb{Z}_{\ge1}\}
\end{equation*}
satisfying the following properties:
\begin{enumerate}
	\item (\emph{Schur field property})
	      For any down-right path $\{(t_i,N_i)\}_{i=1}^{r}$, the joint distribution of the Young diagrams
	      $\lambda^{(i)}=\lambda^{(t_{i+1},N_i)}$ is 
				described by 
				the Schur process
				corresponding to this down-right path.
				Note that this almost surely enforces the boundary conditions
				$\lambda^{(0,N)}=\lambda^{(t,0)}\equiv\varnothing$, and
	      also forces each
				diagram $\lambda^{(t,N)}$ to have at most $N$ parts.
	\item (\emph{PushTASEP coupling property})
	      The collection of random variables
	      \begin{equation}
		      \label{eq:BF_push_coupling_1}
		      \{N-\lambda_1'^{(t,N)}\colon t\in \mathbb{R}_{\ge0},
		      N\in \mathbb{Z}_{\ge1}\}
	      \end{equation}
	      (where $\lambda_1'^{(t,N)}$ is the length
				of the first column of $\lambda^{(t,N)}$)
				has the same
	      distribution as the values of the height function
	      \begin{equation}
		      \label{eq:BF_push_coupling_2}
		      \{h(t,N)\colon t\in \mathbb{R}_{\ge0}, N\in
		      \mathbb{Z}_{\ge1}\}
	      \end{equation}
	      in the inhomogeneous PushTASEP 
				having the speed function $\xi_\bullet$
				and started from the empty initial configuration.
\end{enumerate}

The first property states that a field couples together Schur processes with
different parameters in a particular way, and the second property requires a field to possess
additional structure relating it to the PushTASEP. The random field point of
view was recently useful in \cite{BorodinBufetovWheeler2016},
\cite{BufetovMatveev2017}, \cite{BufetovPetrovYB2017}, \cite{BufetovMucciconiPetrov2018} in 
discovering and studying
particle
systems powered by generalizations of Schur processes.

The above two properties do not determine a field uniquely.
In fact, there exist several constructions of fields satisfying 
these properties.
They lead to \emph{different} joint distributions of all the diagrams
$\{\lambda^{(t,N)}\}$. However, due to the Schur field property, 
along down-right paths the joint distributions of diagrams
are the same.

The oldest known such construction is based on the column RSK insertion.
Connections between RSK, random Young diagrams, and stochastic particle systems
can be traced \cite{Vershik1986}, see also \cite{OConnell2003},
\cite{OConnell2003Trans}. 
Another field coupling Schur processes and the PushTASEP was suggested
in \cite{BorFerr2008DF} based on an idea \cite{DiaconisFill1990} of stitching together
Markov processes connected by a Markov projection operator. A unified
treatment of these two approaches was performed in \cite{BorodinPetrov2013NN},
see also \cite{OConnellPei2012}. 
A variation of the field of
\cite{BorFerr2008DF} based on the Yang-Baxter equation was suggested recently
in \cite{BufetovPetrovYB2017} (for Schur processes, as well as for their
certain two-parameter generalizations),
and further extended in 
\cite{BufetovMucciconiPetrov2018}.
Since either of these approaches
suffices for our purposes, let us outline the simplest 
one from \cite{BorFerr2008DF}.

\medskip

Fix $K\ge 1$, and consider the restriction of the field to
the first $K$ horizontal levels. Interpret $t$ as continuous time, and the
integers $\{\lambda^{(t,N)}_i\colon 1\le N\le K, 1\le i\le N\}$ as a
two-dimensional time-dependent array.\footnote{If a Young diagram
$\lambda^{(t,N)}$ has less than $N$ parts, append it by zeroes.} We will
describe a Markov evolution of this array. Throughout the evolution, the
integers will almost surely satisfy the interlacing constraints
$\lambda^{(t,i)}_{j+1}\le 
\lambda^{(t,i-1)}_j\le
\lambda^{(t,i)}_{j}$
for all $i,j$ and at all times $t$.
These interlacing constraints are visualized in \Cref{fig:new_interlacing}.
\begin{figure}[htpb]
	\centering
	\includegraphics[width=.65\textwidth]{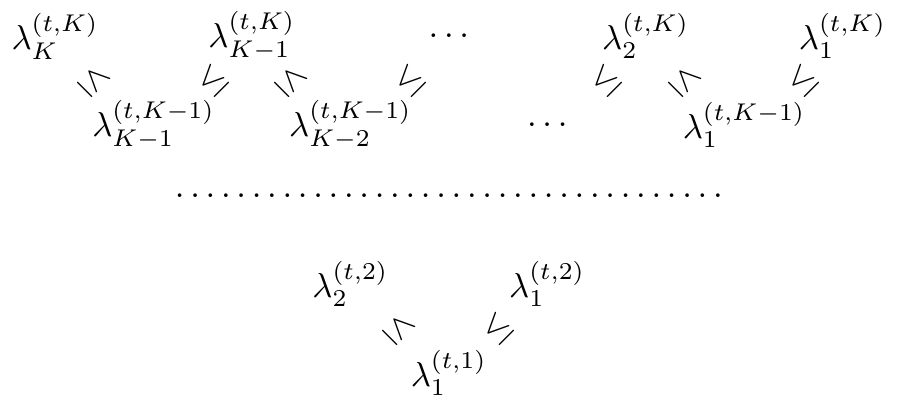}
	\caption{Interlacing array.}
	\label{fig:new_interlacing}
\end{figure}

The array evolves as follows. 
Each of the integers at each level $1\le N\le K$ has an independent
exponential clock with rate $\xi_N$. When the clock of $\lambda^{(t,N)}_j$
rings (almost surely, at most one clock can ring at a given time moment since
the number of clocks is finite), its value is generically incremented by one. In
addition, the following mechanisms are at play to preserve interlacing in the
course of the evolution:
\begin{itemize}
	\item
	      (\emph{blocking}) If
	      $\lambda^{(t,N)}_j=\lambda^{(t,N-1)}_{j-1}$ before the
	      increment of $\lambda^{(t,N)}_j$, then this increment is
	      suppressed;
	\item
	      (\emph{mandatory pushing})
	      If
	      $\lambda^{(t,N)}_j=\lambda^{(t,N+1)}_j=
				\ldots=\lambda^{(t,N+m)}_j $ for some
	      $m\ge 1$ before the increment of $\lambda^{(t,N)}_j$, then along
	      with adding one to $\lambda^{(t,N)}_j$, we also increment by one
	      each of $\lambda^{(t,N+1)}_j,\ldots,\lambda^{(t,N+m)}_j$.
\end{itemize}
Thus described Markov processes are compatible for various $K$, and so they
define a random field $\lambda^{(t,N)}$, $t\in \mathbb{R}_{\ge0}$, $N\in
\mathbb{Z}_{\ge1}$. From \cite{BorFerr2008DF} 
(see also, e.g., \cite[Section 2]{BorodinPetrov2013NN}
for a relatively brief outline of the general formalism)
it follows that the collection of
random Young diagrams $\{\lambda^{(t,N)}\}$
satisfies the Schur field property, i.e., its distributions along
down-right paths are given by Schur processes.

\begin{remark}
	Note that the interlacing inequalities in \Cref{fig:new_interlacing}
	are non-strict, while after the shifting as in \eqref{eq:particle_configuration_for_Schur_process}
	some of these inequalities between consecutive levels become strict.
\end{remark}

\begin{proof}[Proof of the PushTASEP coupling property.]
	Let us now prove that the just constructed collection $\{\lambda^{(t,N)}\}$
	of Young diagrams satisfies the PushTASEP coupling property.
	Observe that $N-\lambda_1'(t,N)$ is the number of zeroes in the $N$-th
	row in the array in \Cref{fig:new_interlacing}. Due to interlacing, 
	for each fixed $t$ we can
	interpret $\tilde{h}(t,N):=N-\lambda_1'(t,N)$ as the height function
	of a particle configuration
	$\tilde{\mathsf{x}}(t)=\{\tilde{\mathsf{x}}_i(t)\}_{i\ge1}$ in
	$\mathbb{Z}_{\ge1}$, with at most one particle per site.
	The initial condition is
	$\tilde{\mathsf{x}}_i(0)=i$, $i\ge1$.
	That is, we can determine
	$\tilde{\mathsf{x}}$ from
	$\tilde{h}$ using \eqref{eq:height_function_def}.

	The time evolution of the particle configuration
	$\tilde{\mathsf{x}}(t)$ is recovered from the field
	$\lambda^{(t,N)}$. 
	First, observe that any change in
	$\tilde{\mathsf{x}}$ can come only from the exponential clocks ringing at the
	rightmost zero elements of the interlacing array. There are two cases. If
	$\tilde{h}(t,N)=\tilde{h}(t,N-1)$, then the rightmost clock
	at zero on level $N$ corresponds to a blocked increment, which agrees with the fact
	that $\tilde{\mathsf{x}}$ has no particle at location $N$. If, on the other
	hand, $\tilde{h}(t,N)=\tilde{h}(t,N-1)+1$, then there is a
	particle in $\tilde{\mathsf{x}}$ at $N$ which can jump to the right by one. This happens at
	rate $\xi_N$. If this particle at $N$ jumps and, moreover,
	$\tilde{h}(t,N+1)=\tilde{h}(t,N)+1$, then the particle at
	$N+1$ which is also present in $\tilde{\mathsf{x}}$ is pushed by one to the right, and so
	on. See \Cref{fig:jumping_push} for illustration.

	\begin{figure}[htpb]
	\centering
	\includegraphics[width=.9\textwidth]{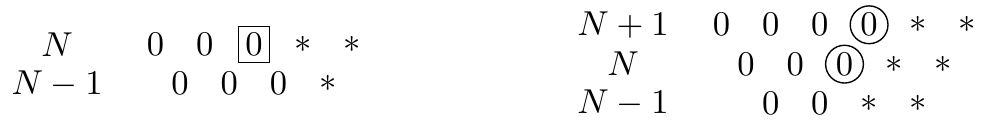}
	\caption{%
		Left: in the interlacing array the framed zero is blocked and cannot
		increase. This corresponds to no particle in $\tilde{\mathsf{x}}(t)$ at $N$.
		Right: the circled zero at level $N$ decides to increase at rate $\xi_N$,
		and forces the circled zero at level $N+1$ to increase, too. 
		In $\tilde{\mathsf{x}}(t)$ this
		corresponds to a jump of the particle at $N$ which then pushes a particle
		at~$N+1$.%
	}
	\label{fig:jumping_push}
	\end{figure}

	We see that the Markov process $\tilde{\mathsf{x}}(t)$ coincides with the
	PushTASEP in inhomogeneous space $\mathsf{x}(t)$ introduced in
	\Cref{sub:PushTASEP_intro}.
\end{proof}

\begin{remark}
	The field $\lambda^{(t,N)}$ from \cite{BorFerr2008DF} 
	described above 
	has \emph{another} Markov projection 
	to a particle system in $\mathbb{Z}$
	which coincides with the
	PushTASEP with \emph{particle-dependent inhomogeneity}. Namely,
	start the PushTASEP from the step initial configuration
	$\mathsf{x}_i(t)=i$, $i\ge1$, and let the \emph{particle} $\mathsf{x}_i$ have jump 
	rate $\xi_i$. The space is assumed homogeneous, so now variable jump rates are
	attached to particles. Then 
	the joint distribution of the random variables $\{\mathsf{x}_i(t)\}$ for all 
	$t\ge0$, $i\ge1$, coincides with the joint distribution
	of $\{\lambda^{(t,i)}_1+i\}$. In particular, 
	each $\mathsf{x}_N(t)$ has the same distribution as $\lambda_1+N$ under the
	Schur measure 
	$\propto s_\lambda(\xi_1,\ldots,\xi_N )s_\lambda(\mathrm{Pl}_t)$ 
	(this is the same Schur measure as in \eqref{eq:Schur_measure_xis}).
	Asymptotic behavior of PushTASEP with particle-dependent
	jump rates was studied 
	in \cite{BG2011non}
	by means of R\'akos--Sch\"utz type determinantal formulas 
	\cite{rakos2005current}, \cite{BorFerr08push}.

	A \emph{third} Markov projection of the field $\lambda^{(t,N)}$ 
	onto 
	$\{\lambda^{(t,N)}_N-N\}_{N\ge1}$
	recovers TASEP on $\mathbb{Z}$ with particle-dependent speeds.
	We refer to \cite{BorFerr2008DF} for details on these other two Markov projections.
\end{remark}

\subsection{From coupling to determinantal structure}
\label{appsub:from_coupling_to_determinants}

For any random field $\lambda^{(t,N)}$
satisfying
the Schur field property,
the 
determinantal structure result of \cite{okounkov2003correlation} recalled in
\Cref{appsub:Schur_correlations} can be restated as follows:
\begin{theorem}
	\label{thm:correlation_kernel_Schur_dynamics}
	For any $m\in \mathbb{Z}_{\ge1}$ and any collection of pairwise
	distinct locations $\{(t_i,N_i,x_i)\}_{i=1}^{m}\subset
		\mathbb{R}\times\mathbb{Z}\times\mathbb{Z}$ such that $N_1\ge \ldots\ge
		N_m\ge0 $ and $0\le t_1\le \ldots\le t_m$, we have
	\begin{multline*}
		\mathbb{P}
		\Bigl(
		\textnormal{for all $i=1,\ldots,m$, the configuration
			$\{\lambda^{(t_i,N_i)}_j-j\}_{j\ge1}$ contains a particle at $x_i$}
		\Bigr)
		\\=
		\det\left[
			K_{\mathsf{F}}(t_p,N_p,x_p;t_q,N_q,x_q)
			\right]_{p,q=1}^{m},
	\end{multline*}
	where
	\begin{equation}
		\label{eq:K_F_Schur_kernel}
		K_{\mathsf{F}}(t,N,x;s,M,y):= \frac{1}{(2\pi \mathbf{i})^2}
		\oint\oint \frac{dz\,dw}{z-w}\frac{w^{y+M}}{z^{x+N+1}} \exp\bigl\{t z-s w\bigr\}
		\frac{\prod_{a=1}^{N}\bigl(z-\xi_a\bigr)}
		{\prod_{b=1}^{M}\bigl(w-\xi_b\bigr)}.
	\end{equation}
	The integration contours are positively oriented simple closed curves around
	$0$, the $w$ contour additionally encircles
	$\{\xi_x\}_{x\in\mathbb{Z}_{\ge1}}$, and the contours satisfy $|z|>|w|$ for
	$t\le s$ and $|z|<|w|$ for $t>s$.
\end{theorem}
In particular, this theorem applies to the field
from \cite{BorFerr2008DF} recalled in \Cref{appsub:Schur_dynamics}
whose first columns are related to the PushTASEP as in 
\eqref{eq:BF_push_coupling_1}--\eqref{eq:BF_push_coupling_2}.

\subsection{Kernel for column lengths}
\label{appsub:column_kernel}

Let us restate
\Cref{thm:correlation_kernel_Schur_dynamics} in terms of column lengths 
so that we can apply it to PushTASEP.

\begin{proposition}
	\label{prop:involution}
	Let $\lambda$ be a Young diagram. The complement in $\mathbb{Z}$ of
	the point configuration $\{\lambda_j-j\}_{j\ge1}$ is the point configuration
	$\{-\lambda_i'+i-1 \}_{i\ge1}$. The former configuration is densely packed at
	$-\infty$, and the latter one at $+\infty$.
\end{proposition}
\begin{proof}
	A straightforward verification, see \Cref{fig:complement_columns}.
\end{proof}

\begin{figure}[htpb]
	\centering
	\includegraphics[width=.6\textwidth]{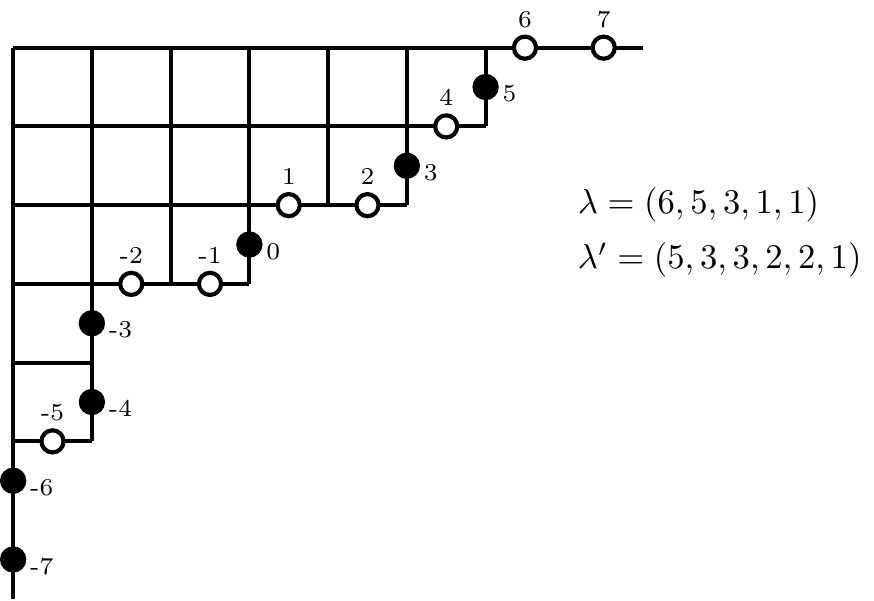}
	\caption{Configuration $\{\lambda_j-j\}$ and its complement
		configuration $\{-\lambda_i'+i-1 \}$, both
		placed at the boundary of the Young diagram $\lambda$.}
	\label{fig:complement_columns}
\end{figure}

The correlation kernel for the complement configuration is given by
$K:=\mathbf{1}-K_{\mathsf{F}}$, where~$\mathbf{1}$ is the identity operator whose
kernel is the delta function. This follows from an observation of S.~Kerov based on the
inclusion-exclusion principle see \cite[Appendix A.3]{Borodin2000b}.
This leads to:

\begin{corollary}
	\label{cor:column_kernel}
	For $\{(t_i,N_i,x_i)\}_{i=1}^{m}\subset
		\mathbb{R}\times\mathbb{Z}\times\mathbb{Z}$ as in
	\Cref{thm:correlation_kernel_Schur_dynamics}, we have
	\begin{equation*}
		\mathbb{P}
		\Bigl(
		\textnormal{$x_i\in\{-\lambda'^{(t_i,N_i)}_j+j-1\}_{j\ge1}$
		for all $i=1,\ldots,m$} \Bigr) = \det\left[
			K(t_p,N_p,x_p;t_q,N_q,x_q)\right]_{p,q=1}^{m},
	\end{equation*}
	where the kernel
	$K(t,N,x;s,M,y):=
	\mathbf{1}_{t=s}
	\mathbf{1}_{N=M}
	\mathbf{1}_{x=y}
	-
	K_{\mathsf{F}}(t,N,x;s,M,y)$
	is given by formula \eqref{eq:K_intro} in the Introduction.
\end{corollary}

\begin{proof}[Proof of \Cref{thm:push_TASEP_multipoint_correlations_intro}]
	This theorem now readily follows from
	\Cref{cor:column_kernel} and
	the PushTASEP coupling property of \Cref{appsub:Schur_dynamics}.
\end{proof}

\subsection{Fredholm determinants}
\label{appsub:Fredholm_dets}

Let us now utilize \Cref{appsub:Schur_dynamics} and \Cref{cor:column_kernel}
to write down observables of the PushTASEP in inhomogeneous space
in terms of Fredholm determinants.

First, recall Fredholm determinants on an abstract discrete space
$\mathfrak{X}$. Let $K(x,y)$, $x,y\in \mathfrak{X}$ be a kernel on this space.
We say that the Fredholm determinant of $\mathbf{1}+zK$, $z\in \mathbb{C}$, is
an infinite series
\begin{equation}
	\label{eq:K_Fredholm_determinant}
	\det(1+zK)_{\mathfrak{X}}=1+\sum_{r=1}^{\infty}
	\frac{z^r}{r!}\sum_{i_1\in\mathfrak{X}}\ldots
	\sum_{i_r\in\mathfrak{X}}
	\det\left[ K(i_p,i_q) \right]_{p,q=1}^{r}.
\end{equation}
One may view \eqref{eq:K_Fredholm_determinant} as a formal series, but in our
setting this series will converge numerically. Details on Fredholm
determinants may be found in \cite{Simon-trace-ideals} or
\cite{Bornemann_Fredholm2010}.

Fix a down-right path $\mathfrak{p}=\{(t_i,N_i)\}_{i=1}^{r}$
and consider the space
\begin{equation*}
	\mathcal{X}=
	\mathcal{X}_1\sqcup \ldots \sqcup\mathcal{X}_r,
	\qquad \mathcal{X}_i =\mathbb{Z}.
\end{equation*}
For $y\in \mathcal{X}_i$ set $t(y)=t_i$, $N(y)=N_i$.
View
$\{ -\lambda_j'^{(t_i,N_i)}+j-1 \}_{j\ge1,\; i=1,\ldots,r }$
as a determinantal process $\mathfrak{L}_{\mathfrak{p}}$ on
$\mathcal{X}$ with kernel $K$ \eqref{eq:K_intro} 
in the sense of \Cref{cor:column_kernel}.

Fix an arbitrary $r$-tuple $\vec{y}=(y_1,\ldots,y_r )\in\mathbb{Z}^{r}$.
We can interpret
\begin{equation*}
	\mathbb{P}\Bigl(
		h(t_i,N_i)> N_i-y_i,
		\ i=1,\ldots,r
	\Bigr)=
	\mathbb{P}\Bigl(
		-\lambda_1'^{(t_i,N_i)}>-y_i,\ 
		i=1,\ldots,r
	\Bigr)
\end{equation*}
as the probability of the event that there are no points
in the random point configuration 
$\mathfrak{L}_{\mathfrak{p}}$
in the subset
$\mathcal{X}_{\vec{y}}:=
\bigsqcup_{i=1}^{r}
\left\{ \ldots,-y_i-2,-y_i-1,-y_i  \right\}$
of $\mathcal{X}$.
This probability can be written (e.g., see \cite{Soshnikov2000}) as the
Fredholm determinant
\begin{equation}\label{eq:Fredholm_short_with_indicator}
	\det(1-\chi_{\vec{y}}
	K\chi_{\vec{y}})_{\mathcal{X}},
\end{equation}
where $\chi_{\vec{y}}(x)=\mathbf{1}_{x\le -y_i}$ for $x\in \mathcal{X}_i$
is the indicator of $\mathcal{X}_{\vec{y}}\subset\mathcal{X}$ viewed
as a projection operator acting on functions. 
In particular, for $r=1$
this implies \Cref{cor:pushTASEP_Fredholm_single_point_formula_intro} from the Introduction.

\begin{remark}
	\label{rmk:Fredholms_finite}
	One can check that the sums in the Fredholm determinant 
	\eqref{eq:Fredholm_short_with_indicator}
	(as well as in \eqref{eq:Fredhom_theorem_intro} in the Introduction)
	are actually finite due to vanishing of $K$ far to the left.
\end{remark}

\section{Asymptotic analysis}
\label{sec:limit_shape_single_fluctuations}

In this section we study asymptotic fluctuations of the random height function
of the inhomogeneous PushTASEP at a single space-time point, and prove
\Cref{thm:limit_shape_result,thm:single_point_fluctuations}. We also establish
more general results on approximating the kernel $K$ \eqref{eq:K_intro} by the
Airy kernel under weaker assumptions on $\boldsymbol\xi(\cdot)$.

\subsection{Rewriting the kernel}
\label{sub:rewrite_kernel}

Let us rewrite $K$ given by
\eqref{eq:K_intro} to make the integration contours
suitable for asymptotic analysis via steepest descent method.

\begin{proposition}
	\label{prop:K_contours_rewriting}
	Let $x'< 0$ and $t'>0$. Then 
	\begin{equation}
		\label{eq:K_contours_rewriting}
		\begin{split}
			&K(t,N,x;t',N',x')
			=
			\mathbf{1}_{t=t'}\mathbf{1}_{N=N'}\mathbf{1}_{x=x'}
			-
			\frac{\mathbf{1}_{t\le t'}}{2\pi\mathbf{i}}
			\oint
			\frac{e^{(t-t')z}dz}{z^{x-x'+N-N'+1}}\prod_{b=N'+1}^N(z-\xi_b)
			\\&\hspace{120pt}-
			\frac{1}{(2\pi\mathbf{i})^2}
			\oint dz \int dw\,
			\frac{e^{tz-t'w}}{z-w}\frac{w^{x'+N'}}{z^{x+N+1}}
			\frac{\prod_{a=1}^{N}(z-\xi_a)}{\prod_{b=1}^{N'}(w-\xi_b)}.
		\end{split}
	\end{equation}
	Here the $z$ contour in both integrals is a positively oriented circle around $0$ (of arbitrary positive radius, say,
	$\delta$),
	and the $w$ contour in the double integral is the vertical line
	$-2\delta+\mathbf{i} \mathbb{R}$
	traversed downwards and located  
	to the left of the $z$ contour.
\end{proposition}
\begin{proof}
	We start from formula \eqref{eq:K_intro} for the kernel.
	For $t\le t'$ (thus necessarily $N\ge N'$ because we consider correlations only
	along down-right paths, cf. \Cref{def:down_right_path})
	the $z$ contour encircles the $w$ contour.
	Note that the integrand does not have poles in $z$ at the $\xi_a$'s.
	Thus, exchanging for $t\le t'$
	the $z$ contour with the $w$ contour at a cost of an additional
	residue, 
	we see that the new contours in the 
	double integral in \eqref{eq:K_intro}
	can be taken as follows:
	\begin{itemize}
		\item 
			the $z$
			contour is a small positive circle around $0$;
		\item
			the $w$ contour is a large positive circle around $0$ and $\{\xi_a\}_{a\ge1}$.
	\end{itemize}
	The additional residue arising for $t\ge t'$
	is equal to the integral of the residue at $z=w$ of the integrand 
	over the single $w$ contour. Because $t\le t'$ and $N\ge N'$, 
	this residue does not have poles at the $\xi_a$'s, and so the integration 
	can be performed over a small contour around $0$.
	Renaming $w$ to $z$ we arrive at the single integral in \eqref{eq:K_contours_rewriting}.

	Finally, in the double integral the $w$ integration 
	contour can be replaced by a vertical line
	because:
	\begin{itemize}
		\item 
			the exponent $e^{-t'w}$ ensures rapid decay of 
			the absolute value of the integrand sufficiently far in the right half plane;
		\item 
			the polynomial factors
			$\frac{w^{x'+N'}}{z-w}\prod_{b=1}^{N'}(w-\xi_b)^{-1}$
			for $x'< 0$ 
			ensure at least quadratic decay of the 
			absolute value of the integrand 
			for sufficiently large $|\Im w|$.
	\end{itemize}
	This completes the proof.
\end{proof}
\begin{remark}
	The assumption $x'< 0$ made in \Cref{prop:K_contours_rewriting}
	agrees with the fact that we are looking at the leftmost 
	points in the determinantal point process
	$\mathfrak{L}_{\mathfrak{p}}$ (\Cref{thm:push_TASEP_multipoint_correlations_intro}),
	and these leftmost points
	almost surely belong to $\mathbb{Z}_{\le0}$.
	At the level of the PushTASEP this corresponds to
	$h(t,N)\le N$. 
	The event $h(t,N)= N$
	(i.e., for which it would be $x'=0$)
	can be excluded, too, since it corresponds to 
	no particles $\le N$ jumping till time $t$. Since $t$
	goes to infinity, this is almost surely impossible.
	We thus assume that $x'< 0$ throughout the text.
\end{remark}

\subsection{Critical points and estimation on contours}
\label{sub:crit_pts_estimates}

Rewrite the integrand in the 
double contour integral in \eqref{eq:K_contours_rewriting}
as 
\begin{equation}\label{eq:integrand_via_S_L}
	\frac{(-1)^{h+h'+N+N'+1}}{z(z-w)}\exp
	\bigl\{ S_L(z;t,N,h)-S_L(w;t',N',h') \bigr\},
\end{equation}
where $h:=x+N$, $h':=x'+N'$, and the function $S_L$ has the form
\begin{equation}
	\label{eq:S_L_action_definition}
	S_L(z;t,N,h):=t z-h\log (-z)+\sum_{a=1}^{N}\log\left( \xi_a-z \right).
\end{equation}
The signs inside logarithms are inserted for future convenience.
The branches of the logarithms are assumed standard, i.e., they have 
cuts along the negative real axis.

We apply the steepest descent
approach (as outlined in, e.g., \cite[Section 3]{Okounkov2002}, in 
a stochastic probabilistic setting)
to analyze the asymptotic behavior of the leftmost points of the
determinantal process $\mathfrak{L}_{\mathfrak{p}}$.
To this end, we consider double critical points of 
$S_L$ which satisfy the following system of equations:
\begin{align}
	\label{eq:S_L_double_critical_equation_t}
	t&=\sum_{a=1}^{N}\frac{\xi_a}{\left( z-\xi_a \right)^2},
	\\
	h&=\sum_{a=1}^{N}\frac{z^2}{\left( z-\xi_a \right)^2}.
	\label{eq:S_L_double_critical_equation_h}
\end{align}

\begin{definition}
	By analogy with \eqref{eq:tau_e_of_eta_defn}, denote 
	$t_e(N):=\sum\limits_{a=1}^N \xi_a^{-1}$.
\end{definition}

In the rest of this subsection we
assume that $N\ge1$ and $0<t<t_e(N)$ are fixed.		

\begin{lemma}
	\label{lemma:root_of_S_L_equation_t_exists}
	Equation \eqref{eq:S_L_double_critical_equation_t}
	has a unique solution in real negative $z$.
\end{lemma}
\begin{proof}
	Follows by monotonicity similarly to
	\Cref{lemma:z_critical_solution_using_t_equation}.
\end{proof}
Denote the solution afforded by \Cref{lemma:root_of_S_L_equation_t_exists}
by $\mathfrak{z}_L=\mathfrak{z}_L(t,N)$.
Also
denote by $\mathfrak{h}_L=\mathfrak{h}_L(t,N)$
the result of substitution of $\mathfrak{z}_L(t,N)$ into 
the right-hand side of \eqref{eq:S_L_double_critical_equation_h}.

\begin{lemma}
	\label{lemma:S_L_double_critical_unique}
	The function $z\mapsto S_L(z;t,N,\mathfrak{h}_L(t,N))$ 
	has a double critical point at 
	$\mathfrak{z}_L(t,N)$
	which is its only critical point on the negative real half-line.
	All other critical points (of any order) 
	of this function are real and positive.
\end{lemma}
\begin{proof}
	The fact that $\mathfrak{z}_L(t,N)$ is a double critical point
	of $S_L(\cdot ;t,N,\mathfrak{h}_L(t,N))$
	follows from 
	the above definitions.
	It remains to check that 
	all other critical points of $S_L$ are real and positive. 
	Let $0<b_1<\ldots<b_k $ be all of the distinct values of $\xi_1,\ldots,\xi_N$.
	Then equation $S_L'(z)=0$, that is,
	\begin{equation}\label{eq:S_L_single_point_equation}
		\frac{h}{z}=t+\sum_{a=1}^{N}\frac{1}{z-\xi_a}
	\end{equation}
	is equivalent to a polynomial equation of degree $k+1$ with real coefficients.
	The right-hand side of \eqref{eq:S_L_single_point_equation}
	takes all values from $-\infty$ to $+\infty$ on each of the $k-1$
	intervals of the form $(b_i,b_{i+1})$.
	Therefore, \eqref{eq:S_L_single_point_equation} has at least $k-1$ positive real roots.
	Since $\mathfrak{z}_L$ is a double root
	when $h=\mathfrak{h}_L$, 
	we have described at least $k+1$ real roots
	to the equation $S_L'(z)=0$, i.e., all of its roots. This completes the proof.
\end{proof}

Keeping $t,N$ fixed, plug $h=\mathfrak{h}_L(t,N)$ into $S_L$,
and 
using \eqref{eq:S_L_double_critical_equation_t}--\eqref{eq:S_L_double_critical_equation_h}
rewrite the result
in terms of $\mathfrak{z}_L$:
\begin{equation*}
	S_L(z;t,N,\mathfrak{h}_L(t,N))=
	\sum_{a=1}^{N}
	\left[ 
		\frac{z\xi_a}{(\mathfrak{z}_L-\xi_a)^2}
		-
		\frac{\mathfrak{z}_L^2\log (-z)}{(\mathfrak{z}_L-\xi_a)^2}
		+
		\log(\xi_a-z)
	\right].
\end{equation*}
Denote the expression inside the sum by $R(z;\xi_a)$.

\begin{lemma}
	\label{lemma:Re_S_L_z_estimate}
	On the circle through $\mathfrak{z}_L$ centered at the origin, 
	$\Re S_L(z;t,N,\mathfrak{h}_L(t,N))$ 
	viewed as a function of $z$
	attains its maximum at $z=\mathfrak{z}_L$.
\end{lemma}
\begin{proof}
	For $z=\mathfrak{z}_L e^{\mathbf{i}\varphi}$, we have
	\begin{equation*}
		\frac{\partial}{\partial \varphi}
		\Re R(\mathfrak{z}_L e^{\mathbf{i}\varphi};\xi)
		=
		\frac{2 \xi ^2 \mathfrak{z}_L^2 (\cos \varphi-1)\sin \varphi }
		{(\mathfrak{z}_L- \xi)^2 \left(\xi^2
		+\mathfrak{z}_L^2-2 \mathfrak{z}_L \xi \cos \varphi\right)}\le0
	\end{equation*}
	for $\varphi\in[0,\pi]$ (by symmetry, it suffices to consider only the upper half plane),
	and this derivative is equal to zero only for $\varphi=0$.
	This implies the claim.
\end{proof}
\begin{lemma}
	\label{lemma:Re_S_L_w_estimate}
	On the vertical line 
	$\mathfrak{z}_L+\mathbf{i}\mathbb{R}$
	through 
	$\mathfrak{z}_L$,
	$\Re S_L(w;t,N,\mathfrak{h}_L(t,N))$
	viewed as a function of $w$
	attains its minimum at $w=\mathfrak{z}_L$.
\end{lemma}
\begin{proof}
	For $w=\mathfrak{z}_L+\mathbf{i}r$, $r>0$, we have
	\begin{equation*}
		\frac{\partial}{\partial r}
		\Re R(\mathfrak{z}_L+\mathbf{i} r;\xi)
		=
		\frac{
			r^3 \left(\xi^2+\mathfrak{z}_L^2-2 \mathfrak{z}_L\xi 
			-\mathfrak{z}_L\right)-r\mathfrak{z}_L (1-\mathfrak{z}_L) (\xi-\mathfrak{z}_L)^2
		}
		{\left(r^2+\mathfrak{z}_L^2\right) (\mathfrak{z}_L-\xi)^2
		\left(\xi^2+r^2+\mathfrak{z}_L^2-2\mathfrak{z}_L\xi\right)}>0
	\end{equation*}
	(recall that $\mathfrak{z}_L<0$). This implies the claim.
\end{proof}

We need one more statement on derivatives of the real part at the 
double critical point:
\begin{lemma}
	\label{lemma:fourth_derivative_nonzero}
	Along the $w$ and $z$ contours 
	in \Cref{lemma:Re_S_L_z_estimate,lemma:Re_S_L_w_estimate}
	the first three derivatives of 
	$\Re S_L$ vanish at $\mathfrak{z}_L$,
	while the fourth derivative is nonzero.
	\end{lemma}
\begin{proof}
	One readily checks that
	$(\frac{\partial}{\partial \varphi})^k
	\Re R(\mathfrak{z}_L e^{\mathbf{i}\varphi};\xi)\,\big\vert_{\varphi=0}=0$
	for $k=1,2,3$, and it is equal to 
	$-\frac{6 \mathfrak{z}_L^2 \xi^2}{(\mathfrak{z}_L-\xi)^4}<0$ for $k=4$.
	The case of the $w$ contour is analogous with a strictly positive fourth
	derivative in $r$ of $\Re R(\mathfrak{z}_L+\mathbf{i} r;\xi)$.
\end{proof}

Let us now deform the integration contours in the double contour integral in
\eqref{eq:K_contours_rewriting} so that they are as in
\Cref{lemma:Re_S_L_z_estimate,lemma:Re_S_L_w_estimate}
(but locally do not intersect at $\mathfrak{z}_L$). 
We can perform this deformation without
picking any residues in particular because the integrand is regular in $z$ at
all the $\xi_a$'s.
\Cref{lemma:Re_S_L_z_estimate,lemma:Re_S_L_w_estimate}
then imply that the asymptotic behavior of the integral 
for large $L$ 
is determined by the contribution
coming from the neighborhood of the double critical point $\mathfrak{z}_L$.
In \Cref{sub:kernel_single_point_approximation} we make precise estimates.

\subsection{Airy kernel}

Before we proceed, let us 
recall the Airy kernel 
\cite{tracy1993level},
\cite{tracy_widom1994level_airy}
\begin{equation}
	\label{eq:usual_Airy}
	\begin{split}
		\mathsf{A}(x;y)
		:=
		\frac{1}{(2\pi\mathbf{i})^2}
		\iint
		\frac{e^{u^3/3-v^3/3-xu+yv}du\,dv}{u-v}
		=\frac{\mathsf{Ai}(x)\mathsf{Ai}'(y)-\mathsf{Ai}'(x)\mathsf{Ai}(y)}{x-y},
	\end{split}
\end{equation}
where 
$x,y\in \mathbb{R}$ (the second expression is extended to $x=y$ by continuity).
In the contour integral expression, the $v$ integration contour goes from
$e^{-\mathbf{i}\frac{2\pi}{3}}\infty$ through $0$ to
$e^{\mathbf{i}\frac{2\pi}{3}}\infty$,
and the $u$ contour goes from
$e^{-\mathbf{i}\frac{\pi}{3}}\infty$ through $0$
to
$e^{\mathbf{i}\frac{\pi}{3}}\infty$,
and the integration contours do not intersect.

The GUE Tracy--Widom distribution function
is the following Fredholm determinant of \eqref{eq:usual_Airy}:
\begin{equation}
	\label{eq:F_2_GUE_definition}
	F_{GUE}(r)=\det\left( \mathbf{1}-\mathsf{A} \right)_{(r,+\infty)},
	\qquad 
	r\in \mathbb{R}.
\end{equation}
Its expansion is
defined analogously to \eqref{eq:K_Fredholm_determinant}
but with sums replaced with integrals over $(r,+\infty)$.

\subsection{Approximation and convergence}
\label{sub:kernel_single_point_approximation}

Our first estimate is a standard approximation of the kernel
$K(t,N,x;t,N,x')$ by the Airy kernel $\mathsf{A}$ \eqref{eq:usual_Airy}
when both $x,x'$ are close to $\mathfrak{h}_L(t,N)-N$.
Denote
\begin{equation}
	\label{eq:d_L_definition}
	\mathfrak{d}_L=
	\mathfrak{d}_L(t,N):=
	\biggl(
			\frac{1}{L}
			\sum_{a=1}^{N}
			\frac{\mathfrak{z}_L^2(t,N)\xi_a}{\left( \xi_a-\mathfrak{z}_L(t,N) \right)^3}
		\biggr)^{1/3}>0.
\end{equation}

In this subsection
we assume that $t=t(L)$ and $N=N(L)$ depend on $L$ such that for all sufficiently large $L$:
\begin{itemize}
	\item 
	$0<t<t_e(N)-c L$ for some $c>0$;
	\item 
	for some $m,M>0$ we have
	$m<\frac{t(L)}L<M$ and $m<\frac{N(L)}L<M$.
\end{itemize}
\begin{lemma}
	\label{prop:K_Airy_single_point_approximation}
	Under our assumptions on $(t(L),N(L))$, 
	as $L\to+\infty$ we have
	\begin{equation}
		\label{eq:K_Airy_single_point_approximation}
		K(t,N,x;t,N,x')=
		\bigl(-\mathfrak{z}_L(t,N)\bigr)^{(\mathsf{h}'-\mathsf{h})L^{1/3}}
		\frac{
		L^{-1/3}}{\mathfrak{d}_L(t,N)}
		\mathsf{A}\biggl(
			-\frac{\mathsf{h}}{\mathfrak{d}_L(t,N)},-\frac{\mathsf{h}'}{\mathfrak{d}_L(t,N)}
		\biggl)
		\bigl(1+O(L^{-1/3})\bigr)
		,
	\end{equation}
	where $x=\mathfrak{h}_L(t,N)-N+\mathsf{h}L^{1/3}$,
	$x'=\mathfrak{h}_L(t,N)-N+\mathsf{h}'L^{1/3}$, 
	$\mathsf{h},\mathsf{h}'\in \mathbb{R}$.
\end{lemma}
\begin{proof}
	When $(t,N)=(t',N')$,
	the indicator and the single contour integral in
	\eqref{eq:K_contours_rewriting}
	cancel out, and so
	we have 
	\begin{equation}
		\label{eq:K_single_proof_1}
		\begin{split}
			K(t,N,x;t,N,x')
			&=
			-
			\frac{1}{(2\pi\mathbf{i})^2}
			\oint dz \int dw\,
			\frac{e^{t(z-w)}}{z(z-w)}\frac{(-w)^{x+N}}{(-z)^{x+N}}
			\prod_{a=1}^{N}\frac{\xi_a-z}{\xi_a-w}
			\\&=
			-
			\frac{1}{(2\pi\mathbf{i})^2}
			\oint dz \int dw\,
			\frac{e^{S_L(z;t,N,h)-S_L(w;t,N,h')}}{z(z-w)},
		\end{split}
	\end{equation}
	where $h=x+N$, $h'=x'+N$.
	Let the $z$ and $w$ integration contours pass near $\mathfrak{z}_L$
	(without intersecting each other)
	and be as in \Cref{lemma:Re_S_L_z_estimate,lemma:Re_S_L_w_estimate}.
	We have
	\begin{multline*}
		S_L(z;t,N,h)-
		S_L(w;t,N,h')
		\\=
		(\mathfrak{h}_L-h)\log(-z)+(h'-\mathfrak{h}_L)\log(-w)
		+
		S_L(z;t,N,\mathfrak{h}_L)-
		S_L(w;t,N,\mathfrak{h}_L).
	\end{multline*}
	For large $L$
	the main contribution to the double 
	integral comes from a small neighborhood of the critical point
	$\mathfrak{z}_L=\mathfrak{z}_L(t,N)$.
	Indeed, fix a neighborhood of $\mathfrak{z}_L$
	of size $L^{-1/6}$. 
	By \Cref{lemma:fourth_derivative_nonzero},
	if $w$ or $z$ or both are
	outside the neighborhood of size $L^{-1/6}$ of $\mathfrak{z}_L$,
	we can estimate
	$\Re (S_L(z;t,N,\mathfrak{h}_L)-S_L(w;t,N,\mathfrak{h}_L))<-c L^{1/3}$ for some $c>0$.
	This means that the contribution coming from outside 
	the neighborhood of $\mathfrak{z}_L$
	is asymptotically negligible
	compared to $(-\mathfrak{z}_L)^{(\mathsf{h}'-\mathsf{h})L^{1/3}}$ 
	in \eqref{eq:K_Airy_single_point_approximation}.

	Inside the neighborhood of $\mathfrak{z}_L$ make a change of variables
	\begin{equation}
		\label{eq:z_w_change_var_single_point_asymptotics}
		z=\mathfrak{z}_L(t,N)+L^{-1/3}\frac{\tilde z}{\mathfrak{c}_L(t,N)},
		\qquad 
		w=\mathfrak{z}_L(t,N)+L^{-1/3}\frac{\tilde w}{\mathfrak{c}_L(t,N)},
	\end{equation}
	where
	\begin{equation}
		\label{eq:c_L_normalization_definition}
		\mathfrak{c}_L(t,N):=
		\Bigl(
			\frac{1}{2L}
			S_L'''(\mathfrak{z}_L(t,N);t,N,\mathfrak{h}_L(t,N))
		\Bigr)^{1/3}
		=
		\biggl(
			\frac{1}{L}
			\sum_{a=1}^{N}
			\frac{\xi_a}{(-\mathfrak{z}_L)\left( \xi_a-\mathfrak{z}_L \right)^3}
		\biggr)^{1/3}
		>0
	\end{equation}
	(so that $\mathfrak{d}_L=-\mathfrak{z}_L\mathfrak{c}_L$).
	Here $\tilde z,\tilde w$ are the scaled integration variables
	which are integrated over the contours in \Cref{fig:Airy_contours}.
	More precisely, $|\tilde z|,|\tilde w|$ go up to order $L^{1/6}$, 
	and the contribution to the Airy kernel $\mathsf{A}$
	coming from the parts of the contours in \eqref{eq:usual_Airy}
	outside this large
	neighborhood of zero
	is bounded from above by $e^{-c L^{1/2}}$ for some $c>0$, and so is asymptotically negligible.

	\begin{figure}[htpb]
		\centering
		\includegraphics{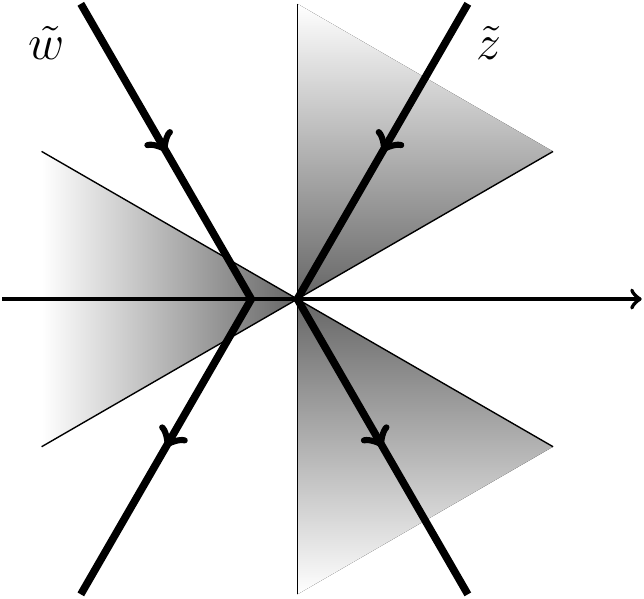}
		\caption{The integration contours for $\tilde z$ and $\tilde w$ in \eqref{eq:z_w_change_var_single_point_asymptotics}
			leading to the Airy kernel approximation.
			Shaded are the regions where $\Re(\tilde z^3)<0$.}
		\label{fig:Airy_contours}
	\end{figure}

	Using \eqref{eq:z_w_change_var_single_point_asymptotics} 
	and Taylor expanding as $L\to+\infty$ we have
	\begin{align*}
		S_L(\mathfrak{z}_L+L^{-1/3}\mathfrak{c}_L^{-1}\tilde z;t,N,\mathfrak{h}_L+\mathsf{h}L^{1/3})
		&=
		S_L(\mathfrak{z}_L;t,N,\mathfrak{h}_L)+\frac{\tilde z^3}{3}-\frac{\mathsf{h}\tilde z}{\mathfrak{z}_L\mathfrak{c}_L}
		\\
		&\hspace{70pt}-L^{1/3} \mathsf{h}\log(-\mathfrak{z}_L)+O(L^{-1/3}),
	\end{align*}
	and similarly for the other term in the exponent in \eqref{eq:K_single_proof_1}.
	Therefore, we have
	\begin{multline*}
		K(t,N,x;t',N',x')=
		\bigl(1+O(L^{-1/3})\bigr)
		(-\mathfrak{z}_L)^{(\mathsf{h}'-\mathsf{h})L^{1/3}}
		\frac{
		L^{-1/3}}{(-\mathfrak{z}_L)\mathfrak{c}_L}
		\\\times
		\frac{1}{(2\pi \mathbf{i})^2}\iint 
		\frac{e^{\tilde z^3/3-\tilde w^3/3
			-
			(\mathfrak{z}_L\mathfrak{c}_L)^{-1}\mathsf{h}\tilde z+(\mathfrak{z}_L\mathfrak{c}_L)^{-1}\mathsf{h}'\tilde w}
		d\tilde z\,d\tilde w}{\tilde z-\tilde w},
	\end{multline*}
	with $\tilde z$, $\tilde w$ contours as in \Cref{fig:Airy_contours}.
	This completes the proof.
\end{proof}
\begin{lemma}
	\label{prop:K_is_small_outside_edge}
	Under our assumptions on $(t(L),N(L))$, 
	let $h=x+N$ and $h'=x'+N$ be such that 
	$h'-\mathfrak{h}_L(t,N)\le -sL^{1/3}$ for some $s>0$.
	Then for some $C,c_1,c_2>0$ and $L$ large enough we have
	\begin{equation*}
		\bigl|K(t,N,x;t,N,x')\bigr|\le 
		C(-\mathfrak{z}_L(t,N))^{h'-h}\cdot
		\frac{ 
		e^{-c_1 L^{1/3}}
		+
		e^{c_2(h'-\mathfrak{h}_L(t,N))L^{-1/3}}
		}{\mathfrak{h}_L(t,N)-h'+1}.
	\end{equation*}
\end{lemma}
\begin{proof}
	First, observe that the assumptions imply that the double critical point
	$|\mathfrak{z}_L(t,N)|$ is 
	uniformly bounded away from zero and infinity.

	Write the kernel 
	as \eqref{eq:K_single_proof_1}
	with integration
	contours described in 
	\Cref{lemma:Re_S_L_z_estimate,lemma:Re_S_L_w_estimate}
	and locally around $\mathfrak{z}_L$ in the proof of \Cref{prop:K_Airy_single_point_approximation}.
	In the exponent we have
	\begin{multline}
		\label{eq:K_outside_edge_proof1}
		\Re\bigl[S_L(z;t,N,h)-S_L(w;t,N,h')\bigr]
		=
		(h'-h)\log|\mathfrak{z}_L|
		+
		(h'-\mathfrak{h}_L)\log|w/\mathfrak{z}_L|
		\\
		+
		\Re\bigl[S_L(z;t,N,\mathfrak{h}_L)
		-
		S_L(w;t,N,\mathfrak{h}_L)\bigr].
	\end{multline}
	If either $z$ or $w$ or both are outside of 
	a $L^{-1/6}$-neighborhood of $\mathfrak{z}_L$, 
	we estimate
	\begin{equation}
		\eqref{eq:K_outside_edge_proof1}
		\le
		(h'-h)\log|\mathfrak{z}_L|
		+
		(h'-\mathfrak{h}_L)\log|w/\mathfrak{z}_L|
		-c L^{1/3}
		\label{eq:K_outside_edge_proof2}
	\end{equation}
	as in the proof of \Cref{prop:K_Airy_single_point_approximation}.
	The part in the exponent containing $w$ is integrable over the vertical $w$ contour,
	which leads to 
	the first term
	in the estimate for $|K|$.

	Now, if both $z,w$ are close to $\mathfrak{z}_L$, make the change of variables
	\eqref{eq:z_w_change_var_single_point_asymptotics} and write
	\begin{equation}
		\eqref{eq:K_outside_edge_proof1}
		\le
		(h'-h)\log|\mathfrak{z}_L|
		+
		(h'-\mathfrak{h}_L)\Bigl[ \frac{L^{-1/3} \Re\tilde w}{\mathfrak{z}_L\mathfrak{c}_L}+O(L^{-2/3}) \Bigr]
		+
		\frac{1}{3}\Re
		\left[ \tilde z^3-\tilde w^3 \right].
		\label{eq:K_outside_edge_proof3}
	\end{equation}
	The part containing $\tilde z,\tilde w$ is integrable over the 
	scaled contours
	in \Cref{fig:Airy_contours}.
	Since the coefficient by $\Re \tilde w$ is positive and $\Re \tilde w\le-1$ on our contour, 
	we estimate this integral using the exponential integral
	$\int_1^{\infty}e^{Au}du=e^{-A}/A$, where $u$ corresponds to $\Re \tilde w$, and
	$A$ is the coefficient by $\Re \tilde w$.
	This produces the 
	second term in the estimate for $|K|$.
	This completes the proof.
\end{proof}

\begin{proposition}
	\label{prop:Fredholm_approximation_Airy}
	Under our assumptions on $(t(L),N(L))$,
	for fixed $\mathsf{y}\in \mathbb{R}$
	and large enough $L$ we have
	\begin{equation}
		\label{eq:prob_Airy_approximation_single_point}
		\mathbb{P}(h(t,N)>\mathfrak{h}_L(t,N)+\mathsf{y}L^{1/3})=
		\bigl(1+O(L^{-1/3})\bigr)
		F_{GUE}\bigl(-\mathsf{y}/\mathfrak{d}_L(t,N)\bigr),
	\end{equation}
	where $\mathfrak{d}_L(t,N)$ is given by 
	\eqref{eq:d_L_definition}.
\end{proposition}
\begin{proof}
	Set $y:=\lfloor\mathfrak{h}_L(t,N)+\mathsf{y}L^{1/3}\rfloor-1$.
	\Cref{cor:pushTASEP_Fredholm_single_point_formula_intro} states that
	the probability in the left-hand side of 
	\eqref{eq:prob_Airy_approximation_single_point}
	is given by a Fredholm determinant with expansion
	\begin{equation}
		\label{eq:Fredholm_det_expansion_in_the_proof}
		1+\sum_{n=1}^{\infty}\frac{(-1)^{n}}{n!}
		\sum_{x_1=-\infty}^{y-N} \ldots \sum_{x_n=-\infty}^{y-N} 
		\;
		\mathop{\mathrm{det}}_{i,j=1}^{n}
		\left[
			K(t,N,x_i;t,N,x_j) 
		\right].
	\end{equation}
	Fix $s>0$ (to be taken large later)
	and separate the terms in the above Fredholm expansion
	where all $x_i>y-N-sL^{1/3}$, plus the remainder. 
	In the former terms we use \Cref{prop:K_Airy_single_point_approximation}, 
	and the latter terms are smaller due to 
	\Cref{prop:K_is_small_outside_edge}.

	When all $x_i>y-N-sL^{1/3}\sim \mathfrak{h}_L-N+\mathsf{y}L^{1/3}-s L^{1/3}$, 
	let us reparametrize the summation variables as
	$x_i=\mathfrak{h}_L-N+\mathsf{u}_i L^{1/3}$, with $\mathsf{u}_i\in \mathbb{R}$
	going from $\mathsf{y}-s$ to $\mathsf{y}$ (in increments of $L^{-1/3}$).
	From \Cref{prop:K_Airy_single_point_approximation} we 
	have\footnote{In general, the kernel $\frac{f(x)}{f(y)}\mathrm{K}(x,y)$
	(with $f$ nowhere vanishing) gives the same determinants as $\mathrm{K}(x,y)$.
	Therefore, 
	the factor $(-\mathfrak{z}_L)^{(\mathsf{h}'-\mathsf{h})L^{1/3}}$
	in \Cref{prop:K_Airy_single_point_approximation},
	as well as the same factor in \Cref{prop:K_is_small_outside_edge},
	do not affect the Fredholm expansion
	and can be ignored.}
	\begin{equation*}
		\mathop{\mathrm{det}}_{i,j=1}^{n}
		\left[
			K(t,N,x_i;t,N,x_j) 
		\right]
		=
		\bigl( 1+O(L^{-1/3}) \bigr)
		\left( \frac{L^{-1/3}}{\mathfrak{d}_L} \right)^{n}
		\mathop{\mathrm{det}}_{i,j=1}^{n}
		\Bigl[
			\mathsf{A}
			\Bigl( 
				-\frac{\mathsf{u}_i}{\mathfrak{d}_L}, 
				-\frac{\mathsf{u}_j}{\mathfrak{d}_L}
			\Bigr)
		\Bigr],
	\end{equation*}
	and each $n$-fold sum
	over $x_{i}>\mathfrak{h}_L-N+(\mathsf{y}-s)L^{1/3}$
	can be approximated (within $O(L^{-1/3})$ error) 
	by the $n$-fold integral of the Airy kernel $\mathsf{A}$
	from $-\mathsf{y}/\mathfrak{d}_L$ to 
	$(s-\mathsf{y})/\mathfrak{d}_L$.
	Taking $s$ sufficiently large
	and using the decay of the Airy kernel
	(e.g., see \cite{tracy_widom1994level_airy})
	leads to the GUE Tracy--Widom distribution function at $-\mathsf{y}/\mathfrak{d}_L$.

	Consider now the remaining terms. 
	Using \Cref{prop:K_is_small_outside_edge} we have
	\begin{multline*}
		(-\mathfrak{z}_L)^{x+N-x'-N'}\sum_{x'<\mathfrak{h}_L-N+(\mathsf{y}-s)L^{1/3}}
		|K(t,N,x;t',N',x')|
		\\
		\le
		C_1e^{-c_1 L^{1/3}}
		\log \bigl(L(s-\mathsf{y})\bigr)+
		(s-\mathsf{y})^{-1}C_2 e^{-c_2(s-\mathsf{y})}
		\bigl(1+O(L^{-1/3})\bigr)
	\end{multline*}
	for some $C_i,c_i>0$.
	The first term decays rapidly for large $L$, 
	and the second term can be made small for fixed $\mathsf{y}$
	by choosing a sufficiently large $s$.

	Take the $n$-th term in 
	\eqref{eq:Fredholm_det_expansion_in_the_proof}
	where 
	some of the $x_i$'s are summed
	from $-\infty$ to $y-N-sL^{1/3}$,
	and
	expand the $n\times n$ determinant along a column $x_j$
	corresponding to 
	$x_j<\mathfrak{h}_L-N+(\mathsf{y}-s)L^{1/3}$.
	The resulting $n-1$ determinants are estimated via 
	H\"older and
	Hadamard's inequalities. Thus,
	the remaining terms in the Fredholm
	expansion 
	are negligible
	and can be included in the error in the right-hand side of 
	\eqref{eq:prob_Airy_approximation_single_point}.
	This completes the proof.
\end{proof}

The final step of the proof of \Cref{thm:single_point_fluctuations}
(which would also imply \Cref{thm:limit_shape_result})
is to show that the approximation of the probability in
\Cref{prop:Fredholm_approximation_Airy}
implies the convergence of the probabilities
$\mathbb{P}(h(t,N)>L\mathfrak{h}(\tau,\eta)+\mathsf{y}L^{1/3})$
to the GUE Tracy--Widom distribution function. 
This convergence would clearly follow if 
\begin{equation}
	\label{eq:h_desired_approximation}
	\mathfrak{h}_L(\tau L,\lfloor \eta L \rfloor )=L\mathfrak{h}(\tau,\eta)+o(L^{1/3}),
	\qquad 
	\mathfrak{d}_L(\tau L,\lfloor \eta L \rfloor )=
	\mathfrak{d}(\tau,\eta)+o(1).
\end{equation}
Observe that all sums in the definitions of $\mathfrak{z}_L$, $\mathfrak{h}_L$, $\mathfrak{d}_L$
in \Cref{sub:crit_pts_estimates}
are Riemann sums of the integrals from $0$ to $\eta$
from \Cref{sub:limit_shape}.
The mesh of these integrals is of order $L^{-1}$, 
and due to the piecewise $C^1$ assumption on $\boldsymbol\xi(\cdot)$,
integrals are approximated by Riemann sums within $O(L^{-1})$.
This implies \eqref{eq:h_desired_approximation},
which is the last step in the proof of \Cref{thm:single_point_fluctuations}.

\appendix

\section{Checking the hydrodynamic equation}
\label{sec:app_check_hydrodynamics}

Here we check that the limiting density 
$\uprho(\tau,\eta)$
defined in \Cref{sub:limit_shape} 
indeed satisfies the hydrodynamic equation
\eqref{eq:hydrodynamic_PDE}. We assume that 
$\eta\ge \tau$ (since $\uprho\equiv0$ clearly satisfies the equation), 
and use formulas
\eqref{eq:critical_solution_t_equation} and
\eqref{eq:h_limit_shape_definition}.
First, note that
the
initial condition
$\uprho(0,\eta)=\mathbf{1}_{\eta\ge0}$
corresponds to the solution $\mathfrak{z}(0,\eta)=-\infty$
of \eqref{eq:critical_solution_t_equation}.

Observe that
\begin{equation*}
	\uprho(\tau,\eta)=
	\frac{\partial}{\partial\eta}\,\mathfrak{h}(\tau,\eta)=
	\frac{\mathfrak{z}^2(\tau,\eta)}
	{
		\left( \mathfrak{z}(\tau,\eta)- \boldsymbol\xi(\eta) \right)^2
	}
	-
	\mathfrak{z}_\eta(\tau,\eta)
	\int_0^\eta
	\frac{2 \mathfrak{z}(\tau,\eta)\boldsymbol\xi(y)}
	{
		\left( \mathfrak{z}(\tau,\eta)- \boldsymbol\xi(y) \right)^3
	}\,dy.
\end{equation*}
The derivative $\mathfrak{z}_\eta$ can be found by differentiating
\eqref{eq:critical_solution_t_equation} in $\eta$:
\begin{equation*}
	0=\frac{\boldsymbol\xi(\eta)}{\left( \mathfrak{z}(\tau,\eta)- \boldsymbol\xi(\eta) \right)^2}
	-
	\mathfrak{z}_\eta(\tau,\eta)
	\int_0^\eta
	\frac{2 \boldsymbol\xi(y)}
	{
		\left( \mathfrak{z}(\tau,\eta)- \boldsymbol\xi(y) \right)^3
	}
	\, dy,
\end{equation*}
which immediately leads to 
$\uprho={\mathfrak{z}}/({\mathfrak{z}- \boldsymbol\xi(\eta)})$,
which is formula 
\eqref{eq:rho_via_z_critical_point} in the Introduction. 
This implies 
\begin{equation*}
	\frac{\boldsymbol\xi(\eta)\uprho(\tau,\eta)}{1-\uprho(\tau,\eta)}
	=
	-
	\mathfrak{z}(\tau,\eta),
	\qquad 
	\frac{\partial}{\partial\eta}
	\left( 
	\frac{\boldsymbol\xi(\eta)\uprho(\tau,\eta)}{1-\uprho(\tau,\eta)}
\right)=-\mathfrak{z}_\eta(\tau,\eta).
\end{equation*}
The remaining term $\uprho_\tau$ in \eqref{eq:hydrodynamic_PDE}
can be expressed through $\mathfrak{z}_\tau$ by
differentiating \eqref{eq:critical_solution_t_equation}
in $\tau$. We have
\begin{equation*}
	1=-
	\mathfrak{z}_\tau(\tau,\eta)
	\int_0^\eta
	\frac{2 \boldsymbol\xi(y)}
	{
		\left( \mathfrak{z}(\tau,\eta)- \boldsymbol\xi(y) \right)^3
	}\, dy,
\end{equation*}
and 
\begin{equation*}
	\uprho_\tau(\tau,\eta)
	=
	\frac{\partial}{\partial\tau}
	\left( \frac{\mathfrak{z}(\tau,\eta)}{\mathfrak{z}(\tau,\eta)- \boldsymbol\xi(\eta)} \right)
	=
	-\mathfrak{z}_\tau(\tau,\eta)
	\,
	\frac{\boldsymbol\xi(\eta)}{\left(\mathfrak{z}(\tau,\eta)- \boldsymbol\xi(\eta)\right)^2}.
\end{equation*}
Combining the above formulas yields the 
hydrodynamic equation \eqref{eq:hydrodynamic_PDE} for 
the limiting density.

\begin{remark}[Homogeneous case]
	\label{rmk:homogeneous_case_density_computation}
	For $\boldsymbol\xi(\eta)\equiv 1$
	equation \eqref{eq:hydrodynamic_PDE} looks as 
	$\tau=\eta/(z-1)^2$, and its unique negative root
	is $\mathfrak{z}(\tau,\eta)=1-\sqrt{\eta/\tau}$, $\eta\ge\tau$ 
	(note that in the homogeneous case $\tau_e(\eta)=\eta$).
	This leads to $\uprho(\tau,\eta)
	={\mathfrak{z}(\tau,\eta)}/({\mathfrak{z}(\tau,\eta)-1)}
	=1-\sqrt{\tau/\eta}$,
	and so the height function is
	$\mathfrak{h}(\tau,\eta)=
	\left( \sqrt\eta-\sqrt\tau \right)^2$,
	as mentioned in \Cref{rmk:homogeneous_case_formulas} in the Introduction.
\end{remark}

\bibliographystyle{amsalpha}
\bibliography{bib}

\bigskip

\textsc{Department of Mathematics, University of Virginia, Charlottesville, VA 22904, and Institute for Information Transmission Problems,
Moscow, Russia 127051}

\textit{E-mail address:} \texttt{lenia.petrov@gmail.com}

\end{document}